\documentclass[11pt,leqno]{amsart}

\usepackage{hyperref}
\usepackage{mathrsfs}
\usepackage{tikz-cd}
\usepackage{amsmath,amsfonts,amssymb}
\usepackage{xfrac}
\usepackage{nicefrac}
\usepackage{color}

\theoremstyle{plain}
\newtheorem{theorem}{Theorem}[section]
\newtheorem{lemma}[theorem]{Lemma}
\newtheorem{proposition}[theorem]{Proposition}
\newtheorem{corollary}[theorem]{Corollary}

\theoremstyle{definition}

\title[The profinite completion of multi-EGS groups]{The profinite completion of  multi-EGS groups}

\author[A. Thillaisundaram]{Anitha Thillaisundaram}

\address{Anitha Thillaisundaram: School of Mathematics and Physics, University of Lincoln,
Brayford Pool, Lincoln LN6 7TS, United Kingdom}
\email{anitha.t@cantab.net}

\author[J. Uria-Albizuri]{Jone Uria-Albizuri}

\address{Jone Uria-Albizuri: Department of Mathematics, University of the Basque Country UPV/EHU, 48940 Leioa, Spain}
 \email{jone.uria@ehu.eus}

\thanks{The first author acknowledges 
support from EPSRC, grant EP/T005068/1, and from the StayConnected Programme at the Heinrich-Heine-Universit\"{a}t; she thanks Heinrich-Heine-Universit\"{a}t for its hospitality. The second author acknowledges financial support from 
 the Spanish Government, grant MTM2014-53810-C2-2-P, and from the Basque Government, grant IT974-16 and the predoctoral grant PRE-2014-1-347. The second author is also supported by the Basque Government through the BERC 2018-2021 program, and by the Spanish Ministry of Science, Innovation and Universities: BCAM Severo Ochoa accreditation SEV-2017-0718.}

\date{\today}
 
\begin{document}
\begin{abstract} 
The class of multi-EGS groups is a generalisation of the well-known Grigorchuk-Gupta-Sidki (GGS-)groups.
Here we classify branch multi-EGS groups with the congruence subgroup property and determine the profinite completion of
all branch multi-EGS 
groups. Additionally our results show that branch multi-EGS groups are just infinite.
\end{abstract}

\maketitle

\section{Introduction}

Let $p$ be an odd prime throughout and let $T$ denote the $p$-adic tree. In the 1980s, Grigorchuk \cite{Grigorchuk}, Gupta 
and Sidki \cite{GuptaSidki} constructed subgroups of the automorphism group $\text{Aut}(T)$ that provided further, and more easily describable, examples of infinite
finitely generated torsion groups, cf. the Burnside problem. These so-called Grigorchuk-Gupta-Sidki groups, or GGS-groups for short, were some of the early examples of branch 
groups. The class of branch groups also contains finitely generated groups with other interesting properties, such as having intermediate 
word growth and being amenable but not elementary amenable \cite{growth}. Just infinite branch groups also  form a natural
partition of the class 
of just infinite groups \cite{Wilson}, where a just infinite group is an infinite group with every proper quotient 
being finite. 

Here we consider multi-EGS groups, which form a generalised family of GGS-groups that also contains Pervova's extended 
Gupta-Sidki (EGS) groups \cite{Pervova}. Pervova's EGS-groups were the first examples of finitely generated branch groups without the congruence 
subgroup property,  that is, when the profinite completion of the group differs from its closure in $\text{Aut}(T)$; see Section 2 for
definitions and details.
The multi-EGS groups were first defined in \cite{KT} 
(though there termed generalised multi-edge spinal groups) and a certain subfamily of them was known to have 
profinite completion differing from the closure in the congruence topology (cf. \cite[Thm.~1.4(3)]{KT}). 
In this paper, we classify the multi-EGS groups which
have the congruence subgroup property. Further we determine the profinite completion of the multi-EGS groups without 
the congruence subgroup property.

Briefly speaking, a multi-EGS group 
\begin{equation} \label{equ:def-gen-mult-ed-sp-gp} G = \big\langle \{a
  \} \cup \{ b^{(j)}_i \mid 1 \leq j \leq p, \, 1 \leq i \leq r_j \}
  \big\rangle,
\end{equation}
where  each $r_j\in \{0,1,\ldots, p-1\}$, is an infinite subgroup of the profinite
group $\mathrm{Aut}(T)$ that is generated by
\begin{itemize}
\item a rooted automorphism $a$ of order $p$,  which cyclically permutes the
  vertices $u_1,\ldots,u_p$ at the $1$st level of~$T$, and
\item families $\mathbf{b}^{(j)} = \{b^{(j)}_1, \ldots,
  b^{(j)}_{r_j}\}$, $j \in \{1,\ldots,p\}$, of directed automorphisms
  sharing a common directed path $P_j$ in~$T$.
\end{itemize}
We require the paths $P_1, \ldots, P_p$ to be mutually disjoint.
The restriction  $0 \leq r_j \leq p-1$, for all $j \in
\{1,\ldots,p\}$, is to ensure that none of the generators
are superfluous. As $G$ is infinite, there is at least one $j \in
\{1,\ldots,p\}$ with $r_j \not = 0$. Each non-empty family $\mathbf{b}^{(j)}$, $j \in \{1,\ldots,p\}$, is defined by a set of vectors 
$\mathbf{E}^{(j)}$, as elaborated in Section 2.

A multi-EGS group is  a
finitely generated, residually-(finite $p$) infinite group.  It is a fractal subgroup of $\mathrm{Aut}(T)$, and from \cite{KT} it is known to be just infinite when torsion. A
restricted subclass of multi-EGS groups was identified in \cite{KT} as being branch; see Section 2 for relevant terminology. As we shall see, this paper identifies all multi-EGS groups that are just infinite and, respectively, branch. It turns out that when a multi-EGS group $G$ is branch, it is furthermore regular branch over $G'$ or $\gamma_3(G)$. 

We  classify  multi-EGS groups $G$ that are branch over $G'$ (compare Proposition~\ref{classification}), and we show that a multi-EGS group $G$ is super strongly fractal if and only if it is branch (compare Proposition~\ref{pro:super}).

Let $\mathscr{E}$ denote the subclass of $3$-generator multi-EGS groups 
$\langle a,b^{(j)},  b^{(k)}\rangle$, for some distinct $j, k\in \{1,\ldots,p\}$, with the associated linearly independent symmetric defining vectors $(e_1,\ldots,e_{p-1})$ and $(f_1,\ldots,f_{p-1})$ satisfying the following condition: subject to replacing the generators $b^{(j)},  b^{(k)}$ with suitable powers, we have that
$e_i, f_i\in \{0,1\}$ and $e_i\ne f_i$ for all $1\le i\le p-1$. Our main results are as follows.

\begin{theorem}\label{MainTheorem}
Let
  $G$
  be a multi-EGS group as in (\ref{equ:def-gen-mult-ed-sp-gp}).
  
 \textup{(A)} Suppose that $G\not \in \mathscr{E}$ is regular branch over $G'$. 
  \begin{enumerate}
  \item[(A.1)] Then $G$ has the congruence subgroup property if and only if the defining vectors
  $\mathbf{E}^{(1)}, \ldots, \mathbf{E}^{(p)}$ are linearly independent. 
  \item[(A.2)] The profinite completion $\widehat{G}$ of $G$ is
 \[
 \widehat{G}=\varprojlim_{n\in \mathbb{N}} 
 G/\big(\psi_n^{-1}(G'\times \overset{\,\,\,p^n}{\cdots} \times G')\big).
 \]
 \end{enumerate}
 Here $\psi_n:\textup{St}_G(n)\longrightarrow G\times \overset{\,\,\,p^n}{\cdots}\times G$ is the natural map under the identification of subtrees rooted at any level $n$ vertex.

  \textup{(B)} Suppose that $G\in \mathscr{E}$ is regular branch over $G'$. 
  
  \begin{enumerate}
  \item[(B.1)]  Then $G$ does not have the congruence subgroup property.
  \item[(B.2)] The profinite completion $\widehat{G}$ of $G$ is
 \[
 \widehat{G}=\varprojlim_{n\in \mathbb{N}} 
 G/\big(\psi_n^{-1}(\gamma_3(G)\times \overset{\,\,\,p^n}{\cdots} \times \gamma_3(G))\big).
 \]
 \end{enumerate}
 
 \textup{(C)} Suppose that $G$ is regular branch over $\gamma_3(G)$ but not over $G'$. Then $G$ has the congruence subgroup property, and hence the profinite completion $\widehat{G}$ of $G$ is equal to the closure of $G$ in $\textup{Aut}(T)$.   
\end{theorem}

The results  of parts (B.1) and (C) are rather unexpected, when compared to part (A.1). The proofs of (A.1), (B.1) and (C) use a similar strategy as was done for the GGS-groups in \cite{FZ, FAGUA, AlejJone}  and for the EGS-groups in~\cite{Pervova}, though there are instances of new methods and ideas. The proof of (A.2)  generalises the 
techniques used in \cite{Pervova}, where the corresponding result was given for torsion EGS-groups.

We further have:
\begin{corollary}\label{justinfinite}
 Let
  $G$ be a multi-EGS group that is
   (regular) branch. Then $G$ is just infinite.
 \end{corollary}

It was shown in \cite{FAGUA} that the GGS-group defined by the constant vector is the  classical GGS-group without the 
congruence subgroup property. 
Now any multi-EGS group $G$ defined by only the constant vector is excluded from Theorem~\ref{MainTheorem} because it is not  branch:

\begin{theorem}\label{constant}
 Let $G$ be a multi-EGS group with constant defining vector. Then $G$ is weakly regular branch but not branch.
\end{theorem}

Lastly, based on the work of Lavreniuk and Nekrashevych~\cite{LN}, we show that, for $G$ a branch multi-EGS group, every automorphism of $G$ is induced by conjugation in $\text{Aut}(T)$.
\begin{theorem}\label{Aut G}
Let $G$ be a branch multi-EGS group. Then $\textup{Aut}(G)=N_{\textup{Aut}(T)}(G)$.
\end{theorem}

\emph{Organisation.} Section 2 of this paper consists of background material for branch groups and multi-EGS groups. Section 3 contains 
preliminary results, the classification of multi-EGS groups that are regular branch over $G'$,  the proof that all multi-EGS groups are super strongly fractal and the proof of Theorem~\ref{Aut G}. In Section 4 we prove parts (A.1), (B.1) and (C) of Theorem \ref{MainTheorem}. In Section 5 we 
prove parts (A.2), (B.2) and  Corollary \ref{justinfinite}. In the final section, we prove Theorem \ref{constant}.

\medskip

\textbf{Acknowledgements.} We gratefully acknowledge the initial involvement of A.~Garrido in the project, and 
we especially thank G.\,A. Fern\'{a}ndez-Alcober for his useful feedback. Furthermore we are grateful to J.~Button,  
\c{S}.~G\"{u}l, B.~Klopsch, B.~Kuckuck, M.~Noce, K.~Rajeev and M.~Vannacci for helpful conversations and we  thank the referee for the helpful comments.

\section{Background material}

 Let $T$ be the $p$-adic tree,
  meaning all vertices   have $p$ children and there is a distinguished vertex called the root.  Using the
  alphabet $X = \{1,2,\ldots,p\}$, the vertices $u_\omega$ of $T$ are
  labelled bijectively by elements $\omega$ of the free
  monoid~$X^*$ in the following natural way.  The root of~$T$
  is labelled by the empty word~$\varnothing$, and for each word
  $\omega \in X^*$ and letter $x \in X$ there is an edge
  connecting $u_\omega$ to~$u_{\omega x}$.  We say
  that $u_\omega$ precedes $u_\lambda$, or equivalently that $u_\lambda$
  succeeds $u_\omega$, when $\omega$ is a prefix of $\lambda$.

 We recall the natural length function on~$X^*$: the words
  $\omega$ of length $\lvert \omega \rvert = n$, which we denote by $X^n$, represent the vertices
  $u_\omega$ that are at distance $n$ from the root. These vertices are called the $n$th
  level vertices and constitute the \textit{$n$th layer} of the tree.
The \textit{boundary} $\partial T$, whose elements correspond
  naturally to infinite simple rooted paths, is in one-to-one correspondence
  with the $p$-adic integers.

 Let $u$ be a vertex of $T$. We denote by $T_u$ the full rooted subtree of $T$ that has  $u$ as its root and includes all vertices succeeding~$u$.  For any
  two vertices $u = u_\omega$ and $v = u_\lambda$, the map
  $u_{\omega \tau} \mapsto u_{\lambda \tau}$, induced by replacing the
  prefix $\omega$ by $\lambda$, yields an isomorphism between the
  subtrees $T_u$ and~$T_v$.  The subtree
  rooted at a generic vertex of level $n$ will be denoted by $T_n$.

  Clearly every $f\in \mathrm{Aut}(T)$ fixes the root and the orbits of
  $\mathrm{Aut}(T)$ on the vertices of  $T$ are precisely its
  layers. We denote the image of a vertex $u$ under
  $f$  by~$u^f$.  The automorphism $f$ induces a faithful action
  on~$X^*$ given by
  $(u_\omega)^f = u_{\omega^f}$.  For $\omega \in X^*$ and
  $x \in X$ we have $(\omega x)^f = \omega^f x'$, for $x' \in X$ 
  uniquely determined by $\omega$ and~$f$.  This induces a permutation
  $f(\omega)$ of $X$ which satisfies
  \[
  (\omega x)^f = \omega^f x^{f(\omega)}, \quad \text{and consequently}
  \quad   (u_{\omega x})^f = u_{\omega^f x^{f(\omega)}}.
  \]
 We say that the automorphism $f$ is \textit{rooted} if $f(\omega) = 1$ for
  $\omega \ne \varnothing$.  It is \textit{directed}, with directed
  path $\ell \in \partial T$, if the support
  $\{ \omega \mid f(\omega) \ne 1 \}$ of its labelling is infinite and
  contains only vertices at distance $1$ from the set of
    vertices corresponding  to the path~$\ell$.
    
When convenient, we do not differentiate between $X^*$ and vertices of~$T$, that is, we do not distinguish between $u_\omega$ and $\omega$, and simply refer to $\omega$ as a vertex of~$T$. For an automorphism~$f$ of~$T$, the \textit{section} of $f$ at a vertex $u$ is the unique automorphism
$f_u$ of $T \cong T_{|u|}$ given by the condition $(uv)^f = u^f
v^{f_u}$ for $v \in X^*$.       
       
    \subsection{Subgroups of $\mathrm{Aut}(T)$}
Let $G$ be a subgroup of $\mathrm{Aut}(T)$. The
\textit{vertex stabiliser} $\mathrm{St}_G(u)$ is the subgroup
consisting of elements in $G$ that fix the vertex~$u$.  For
$n \in \mathbb{N}$, the \textit{$n$th level stabiliser}
  $\mathrm{St}_G(n)= \bigcap_{\lvert \omega \rvert =n}
  \mathrm{St}_G(u_\omega)$
is the subgroup consisting of automorphisms that fix all vertices at
level~$n$.  Let $T_{[n]}$ be the finite subtree of $T$ on
vertices up to level~$n$. Then  $\mathrm{St}_G(n)$ is equal to
the kernel of the induced action of $G$ on $T_{[n]}$.

The full automorphism group 
\[
\mathrm{Aut}(T)= \varprojlim_{n\to\infty} \mathrm{Aut}(T_{[n]})
\]
 is a profinite group, where the topology of $\mathrm{Aut}(T)$ is defined by the open subgroups
$\mathrm{St}_{\mathrm{Aut}(T)}(n)$ for $n \in \mathbb{N}$.  A
subgroup $G$ of $\mathrm{Aut}(T)$ has the \textit{congruence subgroup
  property} if for every subgroup $H$ of finite index in $G$, there
exists some $n$ such that $\mathrm{St}_G(n)\subseteq H$, and we say that $H$ is a \emph{congruence subgroup}. In other words, if the closure of $G$ in $\mathrm{Aut}(T)$ is the same as the profinite completion of $G$.  A weaker version of the congruence subgroup property is the $p$-\emph{congruence subgroup property} for a prime $p$:  a subgroup $G$ of $\mathrm{Aut}(T)$ has the $p$-congruence subgroup
  property if for every subgroup $H$ of finite $p$-power index in $G$, there
exists some $n$ such that $\mathrm{St}_G(n)\subseteq H$ (compare \cite{GUA2}).

Each $g\in \mathrm{St}_{\mathrm{Aut}(T)} (n)$ can be 
   completely determined in terms of its restrictions to the subtrees
  rooted at vertices at level~$n$.  There is a natural
  isomorphism
\[
\psi_n \colon \mathrm{St}_{\mathrm{Aut}(T)}(n) \rightarrow
\prod\nolimits_{\lvert \omega \rvert = n} \mathrm{Aut}(T_{u_\omega})
\cong \mathrm{Aut}(T) \times \overset{p^n}{\cdots} \times
\mathrm{Aut}(T).
\]

Let  $\omega\in X^*$ be of length $n$. We further define 
\[
\varphi_\omega :\mathrm{St}_{\mathrm{Aut}(T)}(u_\omega) \rightarrow \mathrm{Aut}(T_{u_\omega}) \cong \mathrm{Aut}(T)
\]
to be the natural restriction to $T_{u_\omega}$. For $H\le \text{St}_{\text{Aut}(T)}(u_\omega)$, we sometimes write $H_w=\varphi_w(H)$. 

A group $G \leq \text{Aut}(T)$ is said to be \emph{self-similar} if the images under $\varphi_{\omega}$ and $\psi_n$ are contained in $G$ and $G \times \overset{p^n}{\cdots} \times G$, respectively.

Let $G$ be a subgroup of $\mathrm{Aut}(T)$ acting \textit{spherically
  transitively}, that is, transitively on every layer of $T$.
Here the vertex
stabilisers at every level are conjugate under~$G$.  
We say that the group $G$ is \textit{fractal} if $\varphi_\omega(\text{St}_G(u_\omega))=G$ for every $\omega\in X^*$, after the natural identification
of subtrees. Furthermore we say that the group $G$ is \emph{strongly fractal} if $\varphi_x(\text{St}_G(1))=G$ for every $x\in X$, and we say that the group $G$ is \emph{super strongly fractal} if, for each $n\in \mathbb{N}$, we have $\varphi_\omega(\text{St}_G(n))=G$ for every word  $\omega\in X^*$ of length $n$; compare \cite[Def.~2.4]{Jone}.

The \textit{rigid vertex stabiliser} of $u$ in $G$ is the subgroup
$\mathrm{Rist}_G(u)$ consisting of all automorphisms in $G$ that fix
all vertices $v$ of $T$ not succeeding~$u$.
  The \textit{rigid $n$th level stabiliser} is the direct product
  \[
  \mathrm{Rist}_G(n) = \prod\nolimits_{\lvert \omega \rvert = n}
  \mathrm{Rist}_G(u_\omega) \trianglelefteq G
  \]
  of the rigid vertex stabilisers of the vertices at level~$n$.  

We recall that the spherically transitive group $G$ is a \emph{branch
  group}, if $\mathrm{Rist}_G(n)$ has finite index in $G$ for every
$n \in \mathbb{N}$.  
If, in addition, the group $G$ is self-similar and there is a subgroup $1 \not = K \leq \mathrm{St}_G(1)$ with
$K\times \cdots \times K \subseteq \psi_1(K)$ and
$\lvert G : K \rvert < \infty$, then $G$ is said to be \emph{regular
  branch over $K$}.  If in the previous definition the condition $\lvert G : K \rvert < \infty$ is omitted, then $G$ is said to be \emph{weakly regular branch over $K$}. Lastly we note that an infinite group $G$ is
\emph{just infinite} if all its proper quotients are finite, and we 
recall from \cite[Cor.~3.5]{KT} that a torsion multi-EGS group is just infinite.


\subsection{The collection $\mathscr{C}$ of multi-EGS groups}

We recall the definition of multi-EGS groups
from \cite{KT}.
For $j \in \{ 1, \ldots, p \}$ let $r_j \in \{0,1,\ldots,p-1\}$, with
$r_j \not = 0$ for at least one index~$j$. We fix the numerical datum
$\mathbf{E} = (\mathbf{E}^{(1)}, \ldots, \mathbf{E}^{(p)})$, where
each
$\mathbf{E}^{(j)} = (\mathbf{e}^{(j)}_1, \ldots,
\mathbf{e}^{(j)}_{r_j})$
is an $r_j$-tuple of $(\mathbb{F}_p)$-linearly independent
vectors
\[
\mathbf{e}^{(j)}_i = \big( e^{(j)}_{i,1}, \ldots, e^{(j)}_{i,p-1}
\big) \in (\mathbb{F}_p)^{p-1}, \quad i \in \{1, \ldots, r_j
\}.
\]
Write $r=r_1+\cdots +r_p$, and we let $\mathbf{V}$ be the vector space spanned by the $r$ vectors in $\mathbf{E}$.

By $a$ we denote the rooted automorphism, corresponding to the
$p$-cycle $(1 \; 2 \; \cdots \; p) \in \mathrm{Sym}(p)$, that
cyclically permutes the vertices at the first level of~$T$.  We note that
\[
S = \big\{ f \in
  \mathrm{Aut}(T) \mid \forall \omega \in X^* : f(\omega) \in
  \langle a \rangle \big\} \cong \varprojlim_{n \in \mathbb{N}} \,
  C_p \wr \cdots \wr C_p \wr C_p,
\]
the inverse limit of $n$-fold iterated wreath products of $C_p$,
is a Sylow-pro\nobreakdash-$p$ subgroup of $\mathrm{Aut}(T)$.  The
\emph{multi-EGS group} in \emph{standard form}
associated to $\mathbf{E}$ is the group
\begin{align*}
  G = G_\mathbf{E} & = \langle a, \mathbf{b}^{(1)}, \ldots,
  \mathbf{b}^{(p)} \rangle \\
  & = \big\langle \{a \} \cup \{ b^{(j)}_i \mid 1 \leq j \leq p, \, 1
  \leq i \leq r_j \} \big\rangle \leq S,
\end{align*}
where, for each $j \in \{1,\ldots,p\}$, the generator family
$\mathbf{b}^{(j)} = \{b^{(j)}_1, \ldots, b^{(j)}_{r_j}\}$ consists of
commuting directed automorphisms $b^{(j)}_i \in \mathrm{St}_{\mathrm{Aut}(T)}(1)$
along the directed path
\[
\big( \varnothing, (p-j+1), (p-j+1)(p-j+1), \ldots \big) \in \partial T
\]
that satisfy the recursive relations
\[
\psi_1(b^{(j)}_i) = \Big( a^{e^{(j)}_{i,j}}, \ldots, a^{e^{(j)}_{i,p-1}},b^{(j)}_i,
a^{e^{(j)}_{i,1}},\ldots, a^{e^{(j)}_{i,j-1}} \Big).
\]
The vector $\mathbf{e}^{(j)}_i$ is called the \emph{defining vector} of
$b^{(j)}_i$. We say that $\mathbf{e}^{(j)}_i$ is \emph{symmetric} if $e^{(j)}_{i,k}=e^{(j)}_{i,p-k}$ for all $1\le k\le \frac{p-1}{2}$; otherwise $\mathbf{e}^{(j)}_i$ is \emph{non-symmetric}.

A \emph{multi-EGS group} is a subgroup of $\mathrm{Aut}(T)$
that is conjugate to a multi-EGS group in
standard form. We let $\mathscr{C}$ be the class of all such multi-EGS
groups.  Further we define a \emph{multi-GGS group} to be a multi-EGS group where $r_j$ is non-zero for only one index $j$ (cf. \cite{Theofanis}). In particular we write $G_j= \langle a,\mathbf{b}^{(j)} \rangle$ for a multi-GGS group in standard form.

The groups in $\mathscr{C}$ are infinite and act spherically transitively on $T$. Furthermore they are fractal, and by \cite[Lem.~2.5]{Jone}, strongly fractal.

\section{Properties of multi-EGS groups}

\subsection{Branching subgroup}

We begin with a useful property of multi-GGS groups.

\begin{lemma}\label{Lemma 1.2} Let
  $G_{j}= \langle a,\mathbf{b}^{(j)} \rangle \in
  \mathscr{C}$, for some $j\in\{1,\ldots,p\}$,
  be in standard form and suppose that either $\mathbf{b}^{(j)}$ features at least one non-symmetric
    defining vector or $r_j\ge 2$. Then
\[
\psi_1(\gamma_3(G_j))\ge  \psi_1(\textup{St}_{G_j}(1)')=G_j'\times \overset{p}{\cdots}\times  G_j'.
\]
\end{lemma}

\begin{proof}
This follows from \cite[Prop.~3.4]{KT} and  \cite[Lem.~2]{AlejJone}.
\end{proof}

Now we establish  the subclasses of $\mathscr{C}$ that are seen to be regular branch over the derived subgroup, and then we prove that
there are no other such subclasses. In other words, we prove the following in several steps.

\begin{proposition} \label{classification}
Let $G = \langle a, \mathbf{b}^{(1)}, \ldots, \mathbf{b}^{(p)} \rangle
  \in \mathscr{C}$ be in standard form. Then $G$ is regular branch over $G'$ if and only if
   \begin{enumerate}
  \item[(i)]   there is a non-empty family $\mathbf{b}^{(j)}$,  $j \in \{1, \ldots, p\}$, that  features at least one non-symmetric
    defining vector; or
  \item[(ii)]  $\textup{dim }\mathbf{V}\ge 2$.
 \end{enumerate}
\end{proposition}

The second part will be proved in two parts: when there is a non-empty family $\mathbf{b}^{(j)}$,  $j \in \{1, \ldots, p\}$, that has $r_j\ge 2$; or when
 $r_j\in\{0,1\}$ for all $j\in \{1,\ldots,p\}$ with all defining vectors being symmetric and $G$ has at least two linearly independent symmetric defining vectors.

First we identify a collection of exceptional
groups in~$\mathscr{C}$: let $\mathscr{G}$ be the subcollection of
groups that are conjugate in $\mathrm{Aut}(T)$ to 
$\langle a, \mathbf{b}^{(1)}, \ldots, \mathbf{b}^{(p)} \rangle \in
\mathscr{C}$
in standard form, where for  
$j \in \{1,\ldots, p\}$ every non-empty  family
$\mathbf{b}^{(j)} = \{b^{(j)}_1\}$ consists of a single directed
automorphism with constant defining vector
$\mathbf{e}^{(j)}_1 = (1,\ldots,1)$.

\begin{lemma}\label{notderivedsubgroup}
  Let
  $G = \langle a, \mathbf{b}^{(1)}, \ldots, \mathbf{b}^{(p)} \rangle
  \in \mathscr{C}\backslash \mathscr{G}$ be in standard form. Then
  \[
  \psi_1(\gamma_3(\mathrm{St}_G(1))) = \gamma_3(G) \times
  \overset{p}{\cdots} \times \gamma_3(G).
  \]
  In particular,
  \[
  \gamma_3(G) \times \overset{p}{\cdots}\times \gamma_3(G) \subseteq
  \psi_1(\gamma_3(G)),
  \]
  and $G$ is regular branch over $\gamma_3(G)$. 
\end{lemma}
\begin{proof}
 Let
  $G = \langle a, \mathbf{b}^{(1)}, \ldots, \mathbf{b}^{(p)} \rangle
  \in \mathscr{C}\backslash \mathscr{G}$ be in standard form. The case where every non-empty family 
  $\mathbf{b}^{(j)}$, $j\in\{1,\ldots, p\}$, features at least one non-constant defining vector has been settled in \cite[Prop.~3.3]{KT}. So we may assume that there is at least one non-empty family with only the constant defining vector, and also that there is at least one $n\in \{1,\ldots,p\}$ with $\mathbf{b}^{(n)}$ featuring a non-constant defining vector. Without loss of generality, we may further assume that all defining vectors in $\mathbf{b}^{(n)}$ are non-constant. We proceed as in the proof of \cite[Prop.~3.3]{KT}.
  
By spherical transitivity, it suffices to show that
 \[
  \gamma_3(G) \times 1\times \cdots \times 1 \subseteq
  \psi_1(\gamma_3(\mathrm{St}_G(1))).
  \]
  Observe that $\gamma_3(G)$ is generated as a normal subgroup by
  commutators $[g_1,g_2,g_3]$ of group elements $g_1, g_2, g_3$
  ranging over the generating set
  $\{a\} \cup \{ b^{(j)}_i \mid 1 \leq j \leq p, \, 1 \leq i \leq r_j
  \}$. Therefore it suffices to prove, for $k,l,m\in \{1,\ldots,p\}$ with $k\ne l$ and any given 
  $c_j \in \{ b^{(j)}_1, \ldots, b^{(j)}_{r_j} \}$, $j \in \{k,l,m\}$, that
  the elements
  \begin{equation} \label{equ:elements1} ([a, c_k,c_l],1,\ldots,1) ,
    \quad ([a,c_k, a],1,\ldots, 1), 
  \end{equation}
  \begin{equation} \label{equ:elements2}
     \quad([c_k,c_l, a],1,\ldots, 1), \quad ([c_k,c_l,c_m],1,\ldots,1),
  \end{equation}
  are contained in $\psi_1(\gamma_3(\mathrm{St}_G(1)))$.
  
  By conjugation (cf. \cite[Lem.~3.1]{KT}), we may assume throughout that the defining vector of $c_k$ has the form $(1,{e_2},\ldots, e_{p-1})$.
  
  The elements in (\ref{equ:elements2}) are easier to deal with: observe that
   \[
  ([c_k,c_l,c_m],1,\ldots,1) = \psi_1 \big( [c_k^{\, a^k}, c_l^{\,
    a^l}, c_m^{\, a^m}] \big) \in
  \psi_1(\gamma_3(\mathrm{St}_G(1))),
  \]
  and, as
  $\psi_1([c_k^{\, a^k},c_l^{\, a^l}]) = ([c_k,c_l],1,\ldots,1)$, we
  can take $d \in \mathrm{St}_G(1)$ (cf. \cite[Lem.~2.5]{Jone}) such that
  $\psi_1(d) = (a,*,\ldots, *)$, where the symbols $*$ denote
  unspecified elements, to deduce that
  \[
  ([c_k,c_l,a],1,\ldots,1) = \psi_1 \big( [c_k^{\, a^k}, c_l^{\,
    a^l},d] \big) \in \psi_1(\gamma_3(\mathrm{St}_G(1))).
  \]

  We next deal with $([a,c_k, a],1,\ldots, 1)$.    
  If $\mathbf{b}^{(k)}$ features a non-constant defining vector, then 
  $([a,c_k, a],1,\ldots, 1)\in \psi_1(\gamma_3(\mathrm{St}_G(1)))$ follows from the proof of \cite[Prop.~3.3]{KT}. 
  So assume that $\mathbf{b}^{(k)}=\{c_k\}$ features the constant defining vector $(1,\ldots,1)$.
  Then we have
  \[
  \psi_1([c_k^{a^{k-1}},c_k^{a^{k}}])=([a,c_k],1,\ldots,1,[c_k,a]).
  \]
  We consider $\psi_1((b^{(n)}_1)^{a^{n-1}})=(a^{e^{(n)}_1},\ldots,a^{e^{(n)}_{p-1}},b_1^{(n)})$. As $b^{(n)}_1$ has a non-constant defining vector, there exists some $i\in \{1,\ldots,p-2\}$ such that $e^{(n)}_i\ne e^{(n)}_{i+1}$. Then
  \[
  \psi_1((b^{(n)}_1)^{a^{n-1-i}})=(a^{e^{(n)}_{i+1}},*,\ldots,*,a^{e^{(n)}_{i}}).
  \]
Next observe that
  \[
  \psi_1(c_k^{a^{k-2}})=(a,\ldots,a,c_k,a).
  \]
  Hence setting $\tilde{g}=  (b^{(n)}_1)^{a^{n-1-i}} (c_k^{a^{k-2}})^{-e_{i}^{(n)}}$ and taking an appropriate power $g$ of $\tilde{g}$ gives
  \[
  \psi_1(g)=(a,*,\ldots,*,1).
  \]
  
  Thus,
  \[
  \psi_1([c_k^{a^{k-1}},c_k^{a^{k}},g])=([a,c_k,a],1\ldots,1)\in \psi_1(\gamma_3(\mathrm{St}_G(1))).
  \]
  
  It remains to settle $([a, c_k,c_l],1,\ldots,1)$.  Suppose that
  \[
  \psi_1(c_l^{a^{l}})=(c_l,a^{f_1},\ldots,a^{f_{p-1}}).
  \]
 Then 
  \[
  \psi_1(({c_k}^{a^{k+1}})^{-f_{p-1}})=(a^{-f_{p-1}},{c_k}^{-f_{p-1}},a^{-f_{p-1}},\ldots, a^{-f_{p-1}})
  \]
  gives
  \[
  \psi_1(({c_k}^{a^{k+1}})^{-f_{p-1}}c_l^{a^{l}})=(a^{-f_{p-1}}c_l,*,\ldots,*,1).
  \]
  Hence
  \[
  \psi_1([c_k^{a^{k-1}},c_k^{a^{k}},({c_k}^{a^{k+1}})^{-f_{p-1}}c_l^{a^{l}}  ])=([a,c_k,a^{-f_{p-1}}c_l],1,\ldots,1)
  \]
  which is equal to $([a,c_k,c_l],1,\ldots,1)$ modulo $\langle ([a,c_k, a],1,\ldots, 1)\rangle^{G}$.
\end{proof}

Next we have the following result, which extends \cite[Prop.~3.4]{KT}.

\begin{lemma}\label{assoc-multi-GGS}
  Let
  $G = \langle a, \mathbf{b}^{(1)}, \ldots, \mathbf{b}^{(p)} \rangle
  \in \mathscr{C}$ 
  be in standard form and suppose that, for some $n\in \{1,\ldots,p\}$, the  multi-GGS subgroup $G_n=\langle a, \mathbf{b}^{(n)}\rangle$ satisfies the conditions of Lemma~\ref{Lemma 1.2}. Then
  \[
  \psi_1(\mathrm{St}_G(1)') = G' \times
  \overset{p}{\cdots} \times G'.
  \]
  In particular,
  \[
  G' \times \overset{p}{\cdots}\times G' \subseteq
  \psi_1(G'),
  \]
  and $G$ is regular branch over $G'$. 
\end{lemma}
\begin{proof}
By spherical transitivity, it suffices to show that 
\[
G'\times 1\times \cdots \times 1\subseteq \psi_1(G'),
\]
and in particular that
\[
([a,b^{(j)}_l],1,\ldots, 1),\,\,([b^{(j)}_l,b^{(k)}_m],1,\ldots,1)\in \psi_1(G'),
\]
for all distinct $j, k\in\{1,\ldots,p\}$ with $1\le l\le r_j$ and $1\le m\le r_k$ when $r_j,r_k\ne 0$.

The second set of commutators is straightforward to obtain:
\[
\psi_1([(b^{(j)}_l)^{a^j},(b^{(k)}_m)^{a^k}])=([b^{(j)}_l,b^{(k)}_m],1,\ldots,1)
\]
So it remains to show that $([a,b^{(j)}_l],1,\ldots, 1)\in \psi_1(G')$,
for all $j\in\{1,\ldots,p\}$ with $1\le l\le r_j$  when $r_j\ne 0$.

By assumption, we may exclude the case when $\mathbf{b}^{(n)}=\{b^{(n)}\}$ consists of just one directed automorphism with symmetric defining vector. Without loss of generality, by \cite[Lem.~3.2]{KT}, we may assume that $e^{(n)}_{i,1}=1$ for all $1\le i\le r_n$.

As $G_n$ is regular branch over $(G_n)'$, we have $([a,b^{(n)}_i],1,\ldots,1)\in \psi_1(G')$ for $1\le i\le r_n$. Thus
\[
\psi_1([(b^{(n)}_1)^{a^{n-1}},(b^{(j)}_l)^{a^j}])=([a,b^{(j)}_l],1,\ldots,1,[b^{(n)}_1,a^{e^{(j)}_{l,p-1}}]),
\]
together with Lemma \ref{notderivedsubgroup}, enables us to deduce our required result.
\end{proof}

Thirdly, we deal with the case when every  multi-GGS subgroup $G_j$ is only regular branch over $\gamma_3(G_j)$
and not over $(G_j)'$; that is, all defining vectors are symmetric, and each family $\mathbf{b}^{(j)}$ consists of a single generator.

\begin{proposition}\label{twosymmetric}
Let $G=\langle a, \mathbf{b}^{(1)},\ldots, \mathbf{b}^{(p)} \rangle \in \mathscr{C}$ be in standard form
with $r_j\in\{0,1\}$ for all $j\in \{1,\ldots,p\}$. Suppose all defining vectors are symmetric 
and there are at least two linearly independent defining vectors. Then $G$ is regular branch over $G'$ and
\[
 \psi_1(\textup{St}_G(1)')=G'\times \overset{p}{\cdots}\times G'.
\]

\end{proposition}

\begin{proof}
Let $J=\{j\in \{1,\ldots,p\}\mid r_j\ne 0\}$. By assumption, we have $G=\langle \{a\}\cup \{b^{(j)}\mid j\in J\} \rangle$. For $k\in J$, write $(f_1,\ldots,f_{p-1})$ for the defining vector of $b^{(k)}$. By \cite[Lem.~3.2]{KT}, we may assume that $f_1=1$. 


Suppose there exists $j\in J\backslash\{k\}$ such that $b^{(j)}$ is defined by $(e_1,\ldots,e_{p-1})$ with $e_1=0$.
Then
\[
\psi_1([(b^{(k)})^{a^k},(b^{(j)})^{a^{j-1}}])=(1,\ldots,1,[a,b^{(j)}]).
\]
Let $t\in\{2,\ldots,p-2\}$ be such that $e_t\ne 0$. Then for any $l\in J\backslash\{j\}$, we obtain
\[
\psi_1([(b^{(j)})^{a^{j-t}},(b^{(l)})^{a^l}])=([a^{e_t},b^{(l)}],1,\ldots,1,[b^{(j)},a^{*}],1,\ldots, 1),
\]
which enables us to extract $([a,b^{(l)}],1,\ldots,1)$ 
making use of Lemma \ref{notderivedsubgroup}.
Lastly, 
\[
 \psi_1([(b^{(j)})^{a^{j}},(b^{(k)})^{a^k}])=([b^{(j)},b^{(k)}],1,\ldots,1)
\]
shows that all generators of $G'\times 1\times \cdots \times 1$, as a normal subgroup, are obtained.

Suppose now that for all $j\in J$ the defining vector of $b^{(j)}$ has non-zero first component. 
Let $j\in J\backslash \{k\}$ be such that the defining vector $(e_1,\ldots,e_{p-1})$ of $b^{(j)}$ is linearly independent 
from $(f_1,\ldots,f_{p-1})$. By replacing $b^{(j)}$ with an appropriate power, we may assume that $e_1=1$.


Let $t=\min \{i\mid e_i\ne f_i\}>1$. Consequently, 
\[
\psi_1((b^{(j)})^{a^{j}}(b^{(k)})^{-a^{k}})=(b^{(j)}(b^{(k)})^{-1},1,\ldots,1,a^{e_t-f_t},\ldots,a^{e_{p-t}-f_{p-t}},1,\ldots,1),
\]
and, for $m\in \{1,\ldots,p\}$ such that  $(e_t-f_t)m= -e_t$ in $\mathbb{F}_p$, we have
\[
\psi_1((b^{(j)})^{a^{j}}((b^{(j)})^{a^{j}}(b^{(k)})^{-a^{k}})^m)=\qquad\qquad\qquad\qquad\qquad\qquad\qquad
\]
\[
\qquad\qquad(b^{(j)}(b^{(j)}(b^{(k)})^{-1})^m,a,a^{e_2},\ldots,a^{e_{t-1}},1,*,\ldots,*,1,a^{e_{t-1}},\ldots,a^{e_2},a).
\]
Hence, writing $d$ for $((b^{(j)})^{a^{j}}((b^{(j)})^{a^{j}}(b^{(k)})^{-a^{k}})^m$, we obtain
\[
\psi_1([d^{a^t},(b^{(j)})^{a^{j}}(b^{(k)})^{-a^{k}}])\qquad\qquad\qquad\qquad\qquad\qquad\qquad\qquad\qquad
\]
\[
\quad\qquad\qquad=\big[(1,a^{e_{t-1}},\ldots,a^{e_2},a,b^{(j)}(b^{(j)}(b^{(k)})^{-1})^m,a,a^{e_2},\ldots,a^{e_{t-1}},1,*,\ldots,*),
\]
\[
\qquad\qquad\qquad\qquad(b^{(j)}(b^{(k)})^{-1},1,\ldots,1,a^{e_t-f_t},\ldots,a^{e_{p-t}-f_{p-t}},1,\ldots,1)\big]
\]
\[
=(1,\overset{t}{\ldots},1,[b^{(j)}(b^{(j)}(b^{(k)})^{-1})^m,a^{e_t-f_t}],1,\ldots,1).\qquad
\]
Now making use of Lemma \ref{notderivedsubgroup}, together with
\[
\psi_1([(b^{(j)})^{a^{j-1}},((b^{(j)})^{a^{j}}(b^{(k)})^{-a^{k}})^m])=([a,(b^{(j)}(b^{(k)})^{-1})^m],1\ldots,1),
\]
gives, via spherical transitivity,
\[
([b^{(j)},a],1,\ldots,1),
\]
and as above, we are able to get $([a,b^{(l)}],1,\ldots,1)$ for any $l\in J\backslash\{j\}$. 
\end{proof}

We note that in the above, all except one of the defining vectors may be constant. 

Now we rule out all other possibilities for being regular branch over the derived subgroup. 
Let $\mathscr{R}$ be the currently known collection of groups in $\mathscr{C}$ that are seen to be regular branch over
the derived subgroup; that is, the groups defined in Lemma~\ref{assoc-multi-GGS} and Proposition~\ref{twosymmetric}  and their conjugates in $\mathrm{Aut}(T)$.

\begin{lemma}
Let $G\in \mathscr{C}\backslash \mathscr{R}$. Then $G$ is not regular branch over $G'$.
\end{lemma}

\begin{proof}
Suppose $G=\langle a,b^{(i_1)},\ldots, b^{(i_r)}\rangle$, for $r\in \mathbb{N}$ and $i_1,\ldots,i_r\in\{1,\ldots,p\}$, 
is in standard form and not in $\mathscr{R}$. In other words, $b^{(i_1)},\ldots, b^{(i_r)}$ have the same 
symmetric defining vector. 
We write $b$ for $b^{(i_1)}$.

It suffices to show that $([a,b],1,\ldots,1)\not\in \psi_1(\text{St}_G(1))$. In fact, we will establish this result
working modulo $\text{St}_G(2)$. As  
$b\equiv (b^{(i_j)})^{a^{i_j-i_1}}$ mod $\text{St}_G(2)$ for $2\le j\le r$, it suffices to consider the GGS-group 
$G=\langle a,b\rangle$. The result now follows from \cite[Proof of Thm.~3.7]{FZ} and \cite[Thm.~3.7]{FAGUA}.
\end{proof}

Thus the proof of Proposition~\ref{classification} is now complete.

\subsection{Auxiliary results}
 
The following results are necessary for the upcoming sections.

\begin{lemma}\label{lem:subdirect}
Let $G= \langle a,\mathbf{b}^{(1)}, \ldots, \mathbf{b}^{(p)} \rangle\in \mathscr{C}\backslash \mathscr{G}$ be in standard form. 
\begin{enumerate}
    \item [(i)] If $G$ is regular branch over $G'$, then $\psi_1(G')$ is a subdirect product of $G\times \overset{p}{\cdots}\times G$.
    \item [(ii)] If $G$ is regular branch over $\gamma_3(G)$ but not over $G'$, then $\psi_1(\gamma_3(G))$ is a subdirect product of $G\times \overset{p}{\cdots}\times G$.
\end{enumerate}
\end{lemma}
\begin{proof}
(i) We note that there is at least one $j\in\{1,\ldots,p\}$ such that $G_j$ is not a GGS-group defined by just the constant vector. The result then follows from considering \cite[Lem.~4]{AlejJone} in combination with the proof of \cite[Lem.~2.5]{FAGUA}.

(ii) Similarly, this follows from \cite[Lem.~2.6]{FAGUA}.
\end{proof}

\begin{lemma}\label{second_derived}
Let $G=\langle a,\mathbf{b}^{(1)}, \ldots, \mathbf{b}^{(p)} \rangle\in \mathscr{R}$ be in standard form. Then
\[
\psi_1(G'')\ge \gamma_3(G)\times \overset{p}{\cdots}\times \gamma_3(G).
\]
\end{lemma}
\begin{proof}
This follows from Lemma \ref{lem:subdirect}(i) and the fact that $G$ is regular branch over $G'$.
\end{proof}

Let $\mathscr{S}\subseteq\mathscr{C}$ denote the subclass of multi-EGS groups that are conjugate to a multi-EGS group in standard form
$\langle a,\mathbf{b}^{(1)}, \ldots, \mathbf{b}^{(p)}\rangle$, where all defining vectors are symmetric and 
$\mathbf{b}^{(j)}=\{ b^{(j)}\}$ for every non-empty family of directed automorphisms. Further, as mentioned in the introduction, let $\mathscr{E}\subseteq\mathscr{R}$ denote the subclass of multi-EGS groups that are conjugate to a $3$-generator multi-EGS group of standard form
$\langle a,b^{(j)},  b^{(k)}\rangle$, for some distinct $j, k\in \{1,\ldots,p\}$, with linearly independent symmetric defining vectors $(e_1,\ldots,e_{p-1})$ and $(f_1,\ldots,f_{p-1})$ satisfying the following condition: subject to replacing the generators $b^{(j)},  b^{(k)}$ with suitable powers, we have that
 $e_i, f_i\in \{0,1\}$ with $e_i\ne f_i$ for all $1\le i\le p-1$. Note that for $p=3$ the subclass $\mathscr{E}$ is empty.

\begin{proposition}\label{key}
 Let
  $G= \langle a,\mathbf{b}^{(1)},\ldots, \mathbf{b}^{(p)} \rangle \in
  \mathscr{C}\backslash (\mathscr{E}\cup\mathscr{G})$
  be in standard form. Then $\textup{St}_G(1)'\le \gamma_3(G)$.
\end{proposition}

\begin{proof}
 Note that $\text{St}_G(1)$ is normally generated by ${b}^{(j)}_i$ for all $1\le j \le p$ and $1\le i\le r_j$. Further, 
 for $g,h\in G$,
 \begin{align*}
  [(b^{(j)}_i)^g, (b^{(k)}_l)^h]&=[b^{(j)}_i[b^{(j)}_i,g], b^{(k)}_l[b^{(k)}_l,h]]\\
  &\equiv [b^{(j)}_i,b^{(k)}_l]\qquad \text{mod }\gamma_3(G),
 \end{align*}
hence it suffices to show that $[b^{(j)}_i,b^{(k)}_l]\in \gamma_3(G)$ for all $j\ne k$ with $1\le i\le r_j$, $1\le l\le r_k$.

 \underline{Case 1:} $G\not \in \mathscr{S}$.
 
 First suppose that both $b^{(j)}_i$ and $b^{(k)}_l$ are defined by non-symmetric  vectors $(e_1,\ldots ,e_{p-1})$
  and $(f_1,\ldots,f_{p-1})$ respectively. Then, from Lemma 3.1, the subgroup
 $G_j=\langle a, \mathbf{b}^{(j)}\rangle$ satisfies
 \[
  \psi_1^{-1}(G_j'\times \overset{p}\cdots \times G_j')=\text{St}_{G_j}(1)'\le \gamma_3(G_j)
 \]
 and likewise for $G_k$. Hence
 \[
  \psi_1^{-1}(( 1,\ldots,1 , [b^{(j)}_i, a^{f_{p-1}}])) \in \gamma_3(G_j)\le \gamma_3(G),
  \]
 and
 \[
  \psi_1^{-1}( ([a^{e_1}, b^{(k)}_l],1,\ldots,1 ))\in \gamma_3(G_k)\le \gamma_3(G).
  \]
Therefore
 \begin{align*}
  [b^{(j)}_i,b^{(k)}_l]&\equiv [(b^{(j)}_i)^{a^{j-1}},(b^{(k)}_l)^{a^k}]\quad \text{mod }\gamma_3(G)\\
  &=\psi_1^{-1}( ([a^{e_1}, b^{(k)}_l],1,\ldots,1 , [b^{(j)}_i, a^{f_{p-1}}])) \in \gamma_3(G),
 \end{align*}
 as required.

Next suppose  that $b^{(j)}_i$ is defined by a non-symmetric  vector $(e_1,\ldots ,e_{p-1})$
  and  $b^{(k)}_l$ is defined by a symmetric  vector $(f_1,\ldots,f_{p-1})$.
 Then there is some $t\in \{1,\ldots, p-1\}$ such that $e_t\equiv_p \alpha f_t$ and $e_{p-t}\equiv_p\beta f_{p-t}$ with distinct $\alpha,
 \beta \in\{0,1,\ldots, p-1\}$. Then
 \[
  \psi_1((b^{(j)}_i)^{a^j}(b^{(k)}_l)^{-\alpha a^k})=(b^{(j)}_i(b^{(k)}_l)^{-\alpha},*,\ldots, *,1,*,\ldots,*)
 \]
 and 
\[
  \psi_1((b^{(j)}_i)^{a^j}(b^{(k)}_l)^{-\beta a^k})=(b^{(j)}_i(b^{(k)}_l)^{-\beta},*,\ldots, *,1,*,\ldots,*)
 \]
 where, in the first case, the 1 is at position $t+1$ and, in the second case, at position $p-t+1$. Thus
 \begin{align*}
  1&= [(b^{(j)}_i)^{a^j}(b^{(k)}_l)^{-\alpha a^k}, \big( (b^{(j)}_i)^{a^j}(b^{(k)}_l)^{-\beta a^k}\big)^{a^{t}}  ]\\
  &\equiv [b^{(j)}_i, (b^{(k)}_l)^{-\beta}][ (b^{(k)}_l)^{-\alpha}, b^{(j)}_i ] \qquad\qquad\,\,\quad\,\,\, \text{mod }\gamma_3(G)\\
  &\equiv [b^{(j)}_i, b^{(k)}_l]^{\alpha -\beta}\qquad\qquad \qquad\qquad\qquad\qquad \text{mod }\gamma_3(G)
 \end{align*}
and since $G'/\gamma_3(G)$ is of exponent $p$, we are done by taking a suitable power of the above.

 \underline{Case 2:} $G\in \mathscr{S}\backslash (\mathscr{S}\cap \mathscr{R})$.
 
 First observe that 
 \begin{align*}
  [b^{(j)},b^{(k)}]&\equiv [(b^{(j)})^{a^{j}}, (b^{(k)})^{a^k}]\qquad \text{mod }\gamma_3(G)\\
  &=\psi_1^{-1}(([b^{(j)},b^{(k)}],1,\overset{p-1}\ldots,1)),
 \end{align*}
 hence it suffices to show that
 $
  \psi_1^{-1}(([b^{(j)},b^{(k)}],1,\overset{p-1}\ldots,1))\in \gamma_3(G).
$

 As $G \in \mathscr{S}\backslash (\mathscr{S}\cap \mathscr{R})$, we have that $b^{(j)}$ and $b^{(k)}$ are defined by the same non-constant vector.
 Write $(e_1,\ldots, e_{p-1})$ for this defining vector. Without loss of 
 generality (cf. \cite[Lem.~3.1]{KT}), we may suppose that $e_1\ne 0$.
 Let $s\in \{1,\ldots,p-2\}$ be minimal such that $e_s\ne e_{s+1}$. Consider the following two elements:
 \[
  \psi_1([(b^{(j)})^{a^{j-1}},a])=(a^{-e_1}b^{(j)}, a^{e_1-e_2},\ldots, a^{e_{p-2}-e_{p-1}}, (b^{(j)})^{-1} a^{e_{p-1}} )\qquad
 \]
and 
\[
 \psi_1([(b^{(j)})^{a^{j-1}},a]^{a^{-s}})=\qquad\qquad\qquad\qquad\qquad\qquad\qquad\qquad\qquad\qquad\qquad
 \]
 \[
 (a^{e_{s}-e_{s+1}},\ldots , a^{e_{p-2}-e_{p-1}}, (b^{(j)})^{-1} a^{e_{p-1}}, a^{-e_1}b^{(j)}, 
 a^{e_1-e_2},\ldots, a^{e_{s-1}-e_{s}} ).
\]
Now there exists an $m\in \{1,\ldots,p-1\}$ such that $(e_s-e_{s+1})m \equiv_p e_1$. Writing $d$ for 
\[
 ([(b^{(j)})^{a^{j-1}},a]^{a^{-s}})^m[(b^{(j)})^{a^{j-1}},a]\in G',
\]
we see that
\[
 1\equiv \psi_1([d, (b^{(k)})^{a^k} (b^{(j)})^{-a^j}])\equiv
 ([b^{(j)},b^{(k)}],1,\overset{p-1}\ldots,1) \quad \text{mod }\gamma_3(G).
\]
Hence we are done.

\underline{Case 3:} $G\in \mathscr{S}\cap \mathscr{R}$.

Let $J=\{j\in \{1,\ldots,p\} \mid r_j\ne 0\}$.
Suppose $l\in J$ and write $(g_1,\ldots,g_{p-1})$ for the defining vector of $b^{(l)}$. As before, we may assume that $g_1\ne 0$, and let $s(l)\in \{1,\ldots, p-2\}$ be minimal such that $g_{s(l)}\ne g_{s(l)+1}$. Write $m(l)$ for the non-zero element of $\mathbb{F}_p$ satisfying $(g_{s(l)}-g_{s(l)+1})m(l)\equiv_p g_1$. We write 
\[
d(l):=([(b^{(l)})^{a^{l-1}},a]^{a^{-s(l)}})^{m(l)}[(b^{(l)})^{a^{l-1}},a]\in G'.
\]
 
Now if there exists distinct $j, k\in J$ such that the defining vectors of $b^{(j)}$ and $b^{(k)}$ are linearly dependent, then replacing $b^{(k)}$ with a suitable power, we may assume that the two defining vectors are equal. The previous case yields $[b^{(j)},b^{(k)}]\in \gamma_3(G)$ and furthermore
\[
\psi_1([d(l),(b^{(j)})^{a^{j}}(b^{(k)})^{-a^k}])=([b^{(l)},b^{(j)}(b^{(k)})^{-1}],1,\ldots,1)\equiv 1\quad \text{ mod }\psi_1(\gamma_3(G)).
\]    
Hence
\[
([b^{(l)},b^{(j)}],1,\ldots,1)\equiv ([b^{(l)},b^{(k)}],1,\ldots,1) \quad \text{mod }\gamma_3(G).
\]
Therefore we may restrict to the case where the defining vectors associated to  $b^{(j)}$, for  $j\in J$,  are pairwise linearly independent.

Suppose that there are three directed automorphisms $b^{(j)}, b^{(k)}, b^{(l)}$ which have pairwise linearly independent defining vectors. Then without loss of generality, we may assume that
\begin{align*}
\psi_1( (b^{(j)})^{a^j}  )&=(b^{(j)},a,a^{e_2},\ldots, a^{e_{p-2}},a),\\
\psi_1( (b^{(k)})^{a^k}  )&=(b^{(k)},a^{f_1},a^{f_2},\ldots, a^{f_{p-2}},a^{f_{p-1}}),\\
\psi_1( (b^{(l)})^{a^l}  )&=(b^{(l)},a^{g_1},a^{g_2},\ldots, a^{g_{p-2}},a^{g_{p-1}}),
\end{align*}
for some exponents $e_i,f_i,g_i\in \mathbb{F}_p$ subject to the vectors being symmetric and $e_1=1$.

The linear independence of the defining vectors implies that one may find
\[
A=\left(\begin{array}{ccc}
x_1 & x_2 & x_3\\
y_1 & y_2 & y_3\\
z_1 & z_2 & z_3
\end{array}\right)\in \text{GL}(3,p),
\]
such that
\begin{align*}
x&:=((b^{(j)})^{a^j})^{x_1} ( (b^{(k)})^{a^k}  )^{x_2}( (b^{(l)})^{a^l}  )^{x_3}, \\
y&:=((b^{(j)})^{a^j})^{y_1} ( (b^{(k)})^{a^k})^{y_2}  ( (b^{(l)})^{a^l}  )^{y_3}, \\
z&:=((b^{(j)})^{a^j})^{z_1}  ((b^{(k)})^{a^k})^{z_2} ((b^{(l)})^{a^l})^{z_3},
\end{align*}
satisfies
\begin{align*}
\psi_1( x  )&=((b^{(j)})^{x_1} ( b^{(k)} )^{x_2}( b^{(l)}  )^{x_3},a,*,\overset{n}{\ldots}, *,1, *,\overset{m}\ldots, *, 1,*, \ldots, *),\\
\psi_1( y )&=((b^{(j)})^{y_1} (b^{(k)})^{y_2}  ( b^{(l)}  )^{y_3},1,*,\overset{n}\ldots, *,a, *,\overset{m}\ldots, *, 1,*, \ldots, *),\\
\psi_1( z )&=((b^{(j)})^{z_1}  (b^{(k)})^{z_2} (b^{(l)})^{z_3},1,*,\overset{n}\ldots, *,1, *,\overset{m}\ldots, *, a,*, \ldots, *),
\end{align*}
where $n+m+3\le  \frac{p-1}{2}$ and $*$ are unspecified powers of $a$. Now the exponents of the $a$'s in $\psi_1(x),\psi_1(y),\psi_1(z)$ still form symmetric vectors. Hence
\begin{align*}
1= [x,y^{a^{n+m+3}}]&\equiv[x,y^{a}]\equiv [x^{a^{n+2}},y] \quad\text{mod }\gamma_3(G),\\
1= [x,z^{a^{n+2}}] \quad\,&\equiv [x,z^{a}]\quad\text{mod }\gamma_3(G),
\end{align*} 

We have
\begin{align*}
 \psi_1([y,x^{a^{n+2}}])&\equiv ([a,b^{(j)}]^{x_1}[a,b^{(k)}]^{x_2}[a,b^{(l)}]^{x_3},1,\ldots,1) \quad\text{mod }\psi_1(\gamma_3(G)), \\
    \psi_1([x,y^{a}])&\equiv ([a,b^{(j)}]^{y_1}[a,b^{(k)}]^{y_2}[a,b^{(l)}]^{y_3},1,\ldots,1) \quad\text{mod }\psi_1(\gamma_3(G)), \\
     \psi_1([x,z^{a}])&\equiv ([a,b^{(j)}]^{z_1}[a,b^{(k)}]^{z_2}[a,b^{(l)}]^{z_3},1,\ldots,1) \quad\text{mod }\psi_1(\gamma_3(G)). 
\end{align*}
Thus, as~$A$ has full rank, one then deduces that
\[
([a,b^{(j)}],1,\ldots,1), ([a,b^{(k)}],1,\ldots,1), ([a,b^{(l)}],1,\ldots,1)\in \psi_1(\gamma_3(G)),
\]
and the result follows.

\smallskip

Therefore it suffices to consider the case when $G$ has only two directed automorphisms $b^{(j)},b^{(k)}$ with linearly independent symmetric defining vectors.  Again we may assume that 
\begin{align*}
\psi_1( (b^{(j)})^{a^j}  )&=(b^{(j)},a,a^{e_2},\ldots, a^{e_{p-2}},a),\\
\psi_1( (b^{(k)})^{a^k}  )&=(b^{(k)},a^{f_1},a^{f_2},\ldots, a^{f_{p-2}},a^{f_{p-1}}).
\end{align*}
As before, there is a matrix 
\[
B=\left(\begin{array}{cc}
x_1 & x_2\\
y_1 & y_2
\end{array}\right)\in \text{GL}(2,p),
\]
and $x,y\in \text{St}_G(1)$ such that
\begin{align*}
\psi_1( x  )&=((b^{(j)})^{x_1} ( b^{(k)} )^{x_2},a,a^{g_2},\ldots,a^{g_s},1, *,\ldots, *),\\
\psi_1( y )&=((b^{(j)})^{y_1} (b^{(k)})^{y_2} ,1\, ,\quad\overset{s}\ldots\quad , \, 1\,,\,a\, , *, \ldots, *),
\end{align*} 
for some $1\le s\le \frac{p-3}{2}$ and $g_2,\ldots, g_s\in \mathbb{F}_p$.

If there is a $3\le n \le p$ such that the $n$th components of $\psi_1(x)$ and $\psi_1(y)$ are both trivial, then we may argue as in the case of three directed automorphisms above. Hence we assume that no such $n$ exists.

Next, if $s>1$ and there is a $g_i\ne 0,1$ for $2\le i\le s$, then this implies that
\begin{align*}
\psi_1([x,y])&\equiv \psi_1([x^{a^{-1}},y])\qquad\qquad\qquad \qquad\,\,\,\,\quad\text{mod }\psi_1(\gamma_3(G))\\
&\equiv([a,(b^{(j)})^{y_1} (b^{(k)})^{y_2}],1,\ldots,1 )\qquad \,\,\,\,\,\text{mod }\psi_1(\gamma_3(G))\\
&\equiv \psi_1([x^{a^{-i}},y])\qquad\qquad\qquad \qquad\!\!\qquad\text{mod }\psi_1(\gamma_3(G))\\
&\equiv ([a,(b^{(j)})^{y_1} (b^{(k)})^{y_2}]^{g_i},1,\ldots,1 ) \qquad \text{mod }\psi_1(\gamma_3(G)),
\end{align*}
which yields
\[
\psi_1([x,y])\equiv ([a,(b^{(j)})^{y_1} (b^{(k)})^{y_2}],1,\ldots,1 )\equiv 1\qquad\text{mod }\psi_1(\gamma_3(G)).
\]
This, combined with the fact that $\psi_1([x,y])\equiv ([(b^{(j)})^{x_1} (b^{(k)})^{x_2},a],1,\ldots,1 )$ modulo $\psi_1(\gamma_3(G))$, yields $([b^{(j)},a],1,\ldots,1 )\equiv 1$ modulo $\psi_1(\gamma_3(G))$, and the result follows. [Note that the case when $b^{(k)}$, or $b^{(j)}$, is defined by the constant vector can be settled in this manner.]

Thus we are left with the case when $g_2,\ldots, g_s=1$, and for the same reason, for each $n>s+1$, the $n$th components of $\psi_1(x)$ and $\psi_1(y)$ form the set $\{1,a\}$.
However, these are exactly the groups in $\mathscr{E}$, which are excluded from this result.
\end{proof}

\begin{lemma}\label{lem:exception}
Let $G=\langle a, b^{(j)},  b^{(k)}\rangle\in \mathscr{E}$, for some distinct $j, k\in \{1,\ldots,p\}$, be in standard form. Then $[b^{(j)},  b^{(k)}]\not\in \gamma_3(G)$ and in particular, we have  $\textup{St}_G(1)'\not\le \gamma_3(G)$.
\end{lemma}

\begin{proof}
For ease of notation, suppose that $j=1$ and $k=p$, and we write $b=b^{(j)}$ and $c=b^{(k)}$. Further we suppose that $\psi_1(b)=(a,1,\ldots,1,a,b)$ and $\psi_1(c)=(c,1,a,\ldots,a,1)$. The general case follows similarly.

  We first claim that 
\[
\psi_1(G'')\ge \gamma_3(G) \langle x,y,z\rangle \times 1\times \overset{p-1}\cdots \times 1,
\]
where 
\[
x:=[b,a][c,a],\qquad y:=[b,c][a,b],\qquad z:=[c,b][a,c].
\]
Indeed, from Lemma~\ref{second_derived}, we have $\psi_1(G'')\ge \gamma_3(G) \times \overset{p}\cdots \times \gamma_3(G)$. Additionally, modulo $\gamma_3(G) \times \overset{p}\cdots \times \gamma_3(G)$, we have:
\begin{align*}
\psi_1([[a,b],[a,b]^a])
&\equiv (1,[a,b],1,\ldots,1,[b,a]) \\
\psi_1([[a,b],[a,b]^{a^2}])
&\equiv ([a,b],[b,a],1,\ldots,1) 
\end{align*}
By taking suitable products of cyclic permutations of the above two elements, we deduce that $\psi_1(G'')$ contains the elements
\[
(1,\overset{i}\ldots,1 ,[a,b],1,\overset{j}\ldots,1,[b,a], 1,\ldots, 1)
\]
for all $i\in \{0,\ldots,p-2\}$ and $j\in \{0,\ldots, p-i-2\}$.

Similarly, using
\begin{align*}
\psi_1([[a,c],[a,c]^a])
&\equiv ([a,c],1,[c,a],1,\ldots,1)\quad \text{mod }\gamma_3(G) \times \overset{p}\cdots \times \gamma_3(G),\\
\psi_1([[a,c],[a,c]^{a^2}])
&\equiv (1,[c,a],[a,c],1,\ldots,1) \quad \text{mod }\gamma_3(G) \times \overset{p}\cdots \times \gamma_3(G),
\end{align*}
and their cyclic permutations, we see that $\psi_1(G'')$ contains
\[
(1,\overset{i}\ldots,1 ,[a,c],1,\overset{j}\ldots,1,[c,a], 1,\ldots, 1)
\]
for all $i\in \{0,\ldots,p-2\}$ and $j\in \{0,\ldots, p-i-2\}$.

The above elements in combination with
\[
\psi_1([ [a,b],[a,c]^a]) \equiv ([b,a],[c,a],1,\ldots,1)\quad \text{mod }\gamma_3(G) \times \overset{p}\cdots \times \gamma_3(G)
\]
yield the element 
\[
([b,a][c,a],1,\ldots,1)=(x,1,\ldots,1)\in \psi_1(G'').
\]
Furthermore, the elements 
\[
(1,\overset{i}\ldots,1 ,[b,a],1,\overset{j}\ldots,1,[c,a], 1,\ldots, 1),
\]
for all $i\in \{0,\ldots,p-2\}$ and $j\in \{0,\ldots, p-i-2\}$, are also in $\psi_1(G'')$.

Finally, upon considering 
\[
\psi_1([[a,b],[a,c]])\equiv ([c,b][a,c],[a,c],1,\ldots,1,[a,b]) \quad \text{mod }\gamma_3(G) \times \overset{p}\cdots \times \gamma_3(G),
\]
it is straightforward to obtain $(y,1,\ldots,1)$ and $(z,1,\ldots,1)$. The claim is now proved.

Note that if $[b,  c]\in \gamma_3(G)$, we would have 
\[
\psi_1^{-1}(([a,c],1,\ldots,1))=[b,c]\in G''.
\]
So we are done if we show that $\psi_1([b,c])$ modulo $\psi_1(\gamma_3(\text{St}_G(1)))$ is not an element of $\psi_1(G'')$ modulo  $\psi_1(\gamma_3(\text{St}_G(1)))$.

Write $\mu=[a,b]$, $\nu=[a,c]$ and $\zeta=[b,c]$. From our discussion above, it follows that
\[
\frac{\psi_1(G'')}{\psi_1(\gamma_3(\text{St}_G(1)))}=\frac{\langle T_{ij}, U_{ij}, V_k,W_k \rangle \psi_1(\gamma_3(\text{St}_G(1)))}{\psi_1(\gamma_3(\text{St}_G(1)))},
\]
where
\begin{align*}
T_{ij}&=(1,\overset{i-1}\ldots,1,\mu,1,\overset{j-i-1}\ldots,1,\mu^{-1},1,\ldots,1),\\
U_{ij}&=(1,\overset{i-1}\ldots,1,\nu,1,\overset{j-i-1}\ldots,1,\nu^{-1},1,\ldots,1),\\
V_k&=(1,\overset{k-1}\ldots,1,\mu\zeta,1,\ldots,1),\\
W_k&=(1,\overset{k-1}\ldots,1,\nu\zeta^{-1},1,\ldots,1),
\end{align*}
with $1\le i\le p-1$, $i+1\le j\le p$, and $1\le k\le p$.
It is now clear that $(\mu,1,\ldots,1)\not \in\psi_1(G'')/\psi_1(\gamma_3(\text{St}_G(1)))$, and the proof is complete.
\end{proof}

\subsection{Super strongly fractal groups}

We include here a result that is of independent interest. It was established for multi-GGS groups in~\cite[Prop.~2.5.4]{thesis} (see also \cite{Jone}).
\begin{proposition}\label{pro:super}
 Let $G = \langle a, \mathbf{b}^{(1)},\ldots , \mathbf{b}^{(p)}\rangle\in
  \mathscr{C}$
  be in standard form. Then $G$ is super strongly fractal if and only if $G\not \in \mathscr{G}$.
\end{proposition}

\begin{proof}
 Suppose first that $G\not\in \mathscr{G}$.
 Since $G$ is strongly fractal, we have $\varphi_x(\text{St}_G(1))=G$ for every vertex $x\in X$ at level $1$. We will show that $\varphi_u(\text{St}_G(n))=G$ for all vertices $u$ at level $n$, for $n\ge 2$.
 
 Let $K$ denote the usual branching subgroup of $G$, that is, $K=G'$ if $G\in \mathscr{R}$ and $K=\gamma_3(G)$ otherwise. Write $K_n =\psi_n^{-1}(K\times \overset{p^n}\cdots\times K)$. Then
 \[
 \psi_n(K_n)=K\times \overset{p^n}\cdots\times K\subseteq (\text{St}_G(1)\times \overset{p^n}\cdots \times \text{St}_G(1))\cap \psi_n(\text{St}_G(n))=\psi_n(\text{St}_G(n+1)).
 \]
 Therefore $K_n\subseteq \text{St}_G(n+1)$, and for each vertex $u$ at level $n$, we have $K=\varphi_u(K_n)\subseteq \varphi_u(\text{St}_G(n+1))$. Hence, for each $x\in X$, we obtain by Lemma~\ref{lem:subdirect} that
 \[
 G=\varphi_x(K)\subseteq \varphi_{ux}(\text{St}_G(n+1)),
 \]
 so that $\varphi_v(\text{St}_G(n+1))=G$ for every vertex $v$ at level $n+1$.
 
 Now suppose that $G\in \mathscr{G}$, that is, $G=\langle a,b^{(i_1)},\ldots, b^{(i_r)}\rangle$, for $r\in \mathbb{N}$ and $i_1,\ldots,i_r\in\{1,\ldots,p\}$ with constant defining vector. By \cite[Prop.~4.3]{Jone}, we may assume that $r\ge 2$. 
 
 As in \cite[Thm.~2.4(i)]{FZ}, we have $|G/\text{St}_G(2)|=p^{p+1}$, and from \cite{FAGT} (which is a direct generalisation of~\cite[Thm.~2.14]{FZ}), we obtain $|G/\text{St}_G(1)'|=p^{rp+1}$. Certainly $\text{St}_G(1)'\le \text{St}_G(2)$, and hence $|\text{St}_G(2)/\text{St}_G(1)'|=p^{p(r-1)}$.
 
 Without loss of generality, suppose that $i_1<i_2<\cdots < i_r$. For $c_1,\ldots,c_s\in \text{St}_G(1)$ with $s\in \mathbb{N}$, we write $\langle c_1,\ldots,c_s\rangle^{\langle a\rangle}$ to denote $\langle c_i^{a^{j}} \mid 1\le i\le s,\, 0\le j\le p-1\rangle$. 
 
 Observe that
 \begin{align*}
     \frac{\text{St}_G(2)}{\text{St}_G(1)'}&=\frac{\langle \, b^{(i_j)}(b^{(i_k)})^{-a^{i_k-i_j}}  \mid 1\le j,k\le r \rangle^G\, \text{St}_G(1)'}{\text{St}_G(1)'}\\
     &=\frac{\langle \, b^{(i_j)}(b^{(i_{j+1})})^{-a^{i_{j+1}-i_j}}  \mid 1\le j\le r-1 \rangle^G\, \text{St}_G(1)'}{\text{St}_G(1)'}\\
     &=\frac{\langle \, b^{(i_j)}(b^{(i_{j+1})})^{-a^{i_{j+1}-i_j}}  \mid 1\le j\le r-1 \rangle^{\langle a\rangle}\, \text{St}_G(1)'}{\text{St}_G(1)'}\\
     &\cong C_p\times \overset{p(r-1)}\cdots\times C_p.
 \end{align*}
 
 We deduce that for $u\in X$ a first level vertex,
 \[
 \varphi_u(\text{St}_G(2))= \langle b^{(i_j)}(b^{(i_{j+1})})^{-1}  \mid 1\le j\le r-1 \rangle \, G'.
 \]
 Therefore, for $x\in X$,
 \[
\varphi_{ux}(\text{St}_G(2)) = \langle a(b^{(i_j)})^{-1}, [b^{(i_j)},b^{(i_k)}] \mid 1\le j,k \le r \rangle^G.
 \]
 
 We write $N=\langle a(b^{(i_j)})^{-1}, [b^{(i_j)},b^{(i_k)}] \mid 1\le j,k \le r \rangle^G$. From \cite[Prop.~3.9]{KT}, one obtains
 \[
 G/NG' = \langle \overline{a} \rangle\cong C_p,
 \]
 where $\overline{a}$ is the image of $a$ in $G/NG'$. Therefore $N\ne G$, and hence $\varphi_{ux}(\text{St}_G(2))\ne G$, as required.
\end{proof}

Thus, for branch multi-EGS groups, strongly fractal is equivalent to super strongly fractal.


\subsection{Automorphisms of branch multi-EGS groups}

Here we prove Theorem~\ref{Aut G}. By~\cite[Thm.~7.5]{LN}, it suffices to show that all branch multi-EGS groups are saturated. Recall that a group $G\le \text{Aut}(T)$ is \emph{saturated} if for any $n\in \mathbb{N}$ there exists a subgroup $H_n\le \text{St}_G(n)$ that is characteristic in $G$ and $\varphi_v(H_n)$ acts spherically transitively on $T_v$ for all level $n$ vertices $v$.

\begin{proof}[Proof of Theorem~\ref{Aut G}]
If $G$ is regular branch over $G'$, then we set $H_0=G$ and $H_{n+1}=H_n'$. If $G$ is regular branch over $\gamma_3(G)$, then $H_0=G$ and $H_{n+1}=\gamma_3(H_n)$.  By Lemma~\ref{lem:subdirect}, the restriction of $G'$ (respectively $\gamma_3(G)$) on the first level vertices of the tree is the whole group $G$. Hence it follows by induction that the restrictions of $H_n$ on the $n$th level vertices is the whole group $G$ and thus acts spherically transitively on every subtree rooted at an $n$th level vertex.
\end{proof}


\subsection{Normal subgroups in  branch multi-EGS groups}


We first recall, from \cite{KT}, the length
functions on the groups $G \in \mathscr{C}$. Fix a group
$G= \langle a,\mathbf{b}^{(1)}, \ldots, \mathbf{b}^{(p)} \rangle \in
\mathscr{C}$ in standard form and consider the free product
\[
\Gamma= \langle \hat a\rangle * \langle
\widehat{\mathbf{b}}^{(1)} \rangle * \cdots * \langle
\widehat{\mathbf{b}}^{(p)} \rangle
\]
of elementary abelian $p$-groups $\langle \hat a \rangle \cong C_p$
and
$\langle \widehat{\mathbf{b}}^{(j)} \rangle = \langle \hat b_1^{(j)},
\ldots, b_{r_j}^{(j)} \rangle \cong C_p^{\, r_j}$
for $1 \leq j \leq p$.  Note that there is a unique epimorphism
$\pi \colon \Gamma \rightarrow G$ such that $\hat a \mapsto a$
and $\hat b_i^{(j)} \mapsto b_i^{(j)}$ for $1 \leq j \leq p$ and
$1 \leq i \leq r_j$, inducing an epimorphism from
$\Gamma/\Gamma' \cong C_p^{\, 1 + r_1 + \cdots + r_p}$
onto~$G/G'$.  The latter is an isomorphism; see 
\cite[Prop.~3.9]{KT}.
 
Each element
$\hat{g} \in \Gamma$ has a unique reduced form
\[
\hat{g} = \hat a^{\alpha_1}\, w_1 \, \hat a^{\alpha_2} \, w_2 \,
\cdots \, \hat a^{\alpha_l} \, w_l \, \hat a^{\alpha_{l+1}},
\]
where $l \in \mathbb{N} \cup \{0\}$,
$w_1,\ldots, w_l \in \langle \widehat{\mathbf{b}}^{(1)} \cup \cdots
\cup \widehat{\mathbf{b}}^{(p)} \rangle \backslash \{1\}$,
and $\alpha_1, \ldots, \alpha_{l+1} \in \mathbb{F}_p$ such
that $\alpha_i \ne 0$ for $i \in \{2,\ldots,l\}$.

Furthermore, note that for each $i\in \{1,\ldots, l\}$, the element $w_i$ can be
uniquely expressed as
\[
w_i =  \big( \widehat{\mathbf{b}}^{(k(i,1))} \big)^{\boldsymbol{\beta}(i,1)}
\cdots \big( \widehat{\mathbf{b}}^{(k(i,n_i))} \big)^{\boldsymbol{\beta}(i,n_i)},
\]
where $n_i \in \mathbb{N}$,
$k(i,1), \ldots, k(i,n_i) \in \{1,\ldots, p\}$, with
$k(i,m) \ne k(i,m +1)$ for $1 \le m \le n_i-1$, and the exponent vectors
\[
\boldsymbol{\beta}(i,m) = \big( \beta(i,m)_1, \ldots,
\beta(i,m)_{r_{k(i,m)}} \big) \in
(\mathbb{F}_p)^{r_{k(i,m)}} \backslash \{\mathbf{0}\}, 
\]
for $1 \le m \le n_i$, are such that
\[
\big( \widehat{\mathbf{b}}^{(k(i,m))} \big)^{\boldsymbol{\beta}(i,m)} = \big(
\hat b^{(k(i,m))}_1 \big)^{\beta(i,m)_1} \cdots \big(
\hat b^{(k(i,m))}_{r_{k(i,m)}} \big)^{\beta(i,m)_{r_{k(i,m)}}}.
\]

The \emph{length} of $\hat{g}$ is defined as
$
\partial(\hat{g}) =  n_1 + \cdots + n_l$.

Let $G \in \mathscr{C}$ and $\pi \colon \Gamma \rightarrow G$ be
the natural epimorphism as above. The \emph{length} of $g\in G$ is then
$
\partial(g) = \min \{\partial(\hat{g}) \mid \hat{g}\in \pi^{-1}(g) \}$.

\bigskip

The following result is important for Section 5. Recall the notation $H_u=\varphi_u(H)$ for $H\le \text{St}_{\text{Aut}(T)}(u)$.

\begin{proposition}\label{Proposition 2.1}
Let $G = \langle a, \mathbf{b}^{(1)},\ldots , \mathbf{b}^{(p)}\rangle\in
  \mathscr{C}\backslash \mathscr{G}$
  be in standard form. Then for any non-trivial $x\in G$  there is a vertex $u$ such that $\textup{St}_N(u)_u=G$   where $N$ is the normal closure of $x$ in $G$.
\end{proposition}

\begin{proof}
It suffices to find a vertex $v$ such that $a\in \text{St}_N(v)_v$ and $a^*b^{(j)}_l\in\text{St}_N(v)_v$ for all $1\le j\le p$ and $1\le l\le r_j$, where $*$ represents unknown exponents. We note that the result is true if $r_j\ne 0$ for only one $j\in \{1,\ldots,p\}$ and this $r_j=1$, by~\cite{Pervova}, so we will inherently exclude this case.

We proceed by induction on the length $\partial(x)$ of $x$ in $G$, and we will make use of the fact that $G$ is fractal without special mention.

\bigskip 

\underline{Case 1:} Suppose $\partial(x) =0$. 

Note that the non-trivial elements of length 0 are of the form $a^i$ for $i\in \{1,\ldots,p-1\}$.   We choose $1\le j\le p$ 
and $1\le l \le r_j$ such that $b_l^{(j)}$ is not defined by the constant vector. Recall that
\[
\psi_1(b_l^{(j)})= (a^{e^{(j)}_{l,j}  },\ldots,  a^{e^{(j)}_{ l,p-1} }, {b^{(j)}_l}, a^{e^{(j)}_{l,1} },\ldots, a^{e^{(j)}_{l,j-1}}    )
\]
where $b^{(j)}_l$ appears in the $(p-j+1)$st coordinate, and
\[
\psi_1(a^{-i}(b^{(j)}_l)^{-1}a^i)=(a^{-e^{(j)}_{l,j-i}},\ldots,a^{-e^{(j)}_{l,p-1}}, (b^{(j)}_l)^{-1},a^{-e^{(j)}_{l,1}},\ldots, a^{-e^{(j)}_{l,j-i-1} }   )
\]
has $(b^{(j)}_l)^{-1}$ in the $(p-j+i+1)$st coordinate. Thus
\[
\psi_1([a^i,b^{(j)}_l])=(a^{f_1},\ldots, a^{f_{p-j}}, 
{a^{f_{p-j+1}}b^{(j)}_l}, a^{f_{p-j+2}},\ldots,\qquad\qquad \qquad\qquad
\]
\[
\qquad\qquad \qquad \qquad\qquad \qquad a^{f_{p-j+i}},
(b^{(j)}_l)^{-1}a^{f_{p-j+i+1}},a^{f_{p-j+i+2}},\ldots, a^{f_{p-1}}),
\]
where
\begin{align*}
 f_1= e^{(j)}_{l,j}-e^{(j)}_{l,j-i},\,&\ldots\,, f_{p-j}=e^{(j)}_{l,p-1}-e^{(j)}_{l,p-i-1},\\
 f_{p-j+1}&=-e^{(j)}_{l,p-i},\\
 f_{p-j+2}= e^{(j)}_{l,1}-e^{(j)}_{l,p-i+1},\, &\ldots\,,\,f_{p-j+i}=e^{(j)}_{l,i-1}-e^{(j)}_{l,p-1},\\
 f_{p-j+i+1}&=e^{(j)}_{l,i},\\
 f_{p-j+i+2}=e^{(j)}_{l,i+1}-e^{(j)}_{l,1},\, &\ldots\,,\, f_{p-1}=e^{(j)}_{l,j-1}-e^{(j)}_{l,j-i-1}.
\end{align*}

If one of the above $f_k$ for $k\in \{1,\ldots,p\}
\backslash \{ p-j+1,p-j+i+1\}$ is non-zero, then an appropriate conjugation shows that $\text{St}_N(v)_v$ contains a 
non-trivial power of $a$, for $v=u_p$. In order to get 
$a^*b^{(j)}_l\in \text{St}_N({v})_{{v}}$, we consider  $[a^i,b^{(j)}_l]^{a^{j-1}}$.

If all of the above mentioned $f_k$'s are zero,  we have that $b^{(j)}_l$ is defined by the 
constant vector, which contradicts the choice of $b^{(j)}_l$. Hence for $v=u_p$, we have 
$\langle a,b^{(j)}_l\rangle \le \text{St}_N(v)_v$. By considering $[a^i,b^{(k)}_m]^{a^{k-1}}$ for $1\le k\le p$ and 
$1\le m \le r_k$, we obtain $G=\text{St}_N(v)_v$, as required.

\bigskip

\underline{Case 2:} Suppose $\partial(x)=1$.

Then we have  $x=a^{\lambda}b^{(k)}a^{\mu}$
for $\lambda,\mu\in \mathbb{F}_p$ and $b^{(k)}\in \langle \mathbf{b}^{(k)}\rangle$ for $k\in \{1,\ldots, p\}$ with $r_k\ne 0$. Conjugating by $a^{\lambda}$, we may assume
$x=b^{(k)}a^{\mu}$.

Suppose $\mu=0$. We form $\widetilde{x}$ by conjugating $x$ by an appropriate power of $a$ so that 
$\varphi_p(\widetilde{x})=a^i$
for some $i\in \{1,\ldots,p-1\}$. Then we consider $[\widetilde{x},(b^{(j)}_l)^{a^{j-1}}]$, with $j$ and $l$ as before, which gives
\[
\varphi_p([\widetilde{x},(b^{(j)}_l)^{a^{j-1}}])=[a^i,b^{(j)}_l]
\] 
and we proceed as in Case 1; remembering that as $G$ is fractal, for every $s\in \mathbb{F}_p$ there is $g_1\in \mathrm{St}_G(1)$ such that $\psi_1(g_1)=(*,\ldots,*,a^s)$.

Next suppose  $\mu\ne 0$. Now consider
\begin{align*}
\psi_1(((b^{(j)}_l)^{-a^{j-k}})^x)&=\psi_1(a^{-\mu}(b^{(k)})^{-1}(b^{(j)}_l)^{-a^{j-k}}b^{(k)}a^{\mu})\\
&=(a^{-e^{(j)}_{l,k-\mu} },\ldots,a^{-e^{(j)}_{l,p-1} },(b^{(k)})^{-1}(b^{(j)}_l)^{-1}b^{(k)},a^{-e^{(j)}_{l,1} },\ldots,a^{-e^{(j)}_{l,k-\mu-1} }),
\end{align*}
where $(b^{(k)})^{-1}(b^{(j)}_l)^{-1}b^{(k)}$ is in the $(p-k+\mu+1)$st coordinate.
Then
\[
\psi_1([x,(b^{(j)}_l)^{a^{j-k}}]) =(a^{\tilde{f}_1},\ldots, a^{\tilde{f}_{p-k}},a^{\tilde{f}_{p-k+1}} b^{(j)}_l ,
a^{\tilde{f}_{p-k+2}},\ldots, \qquad\qquad \qquad\qquad
\]
\[
\qquad\qquad \qquad \qquad\qquad \qquad 
a^{\tilde{f}_{p-k+\mu}},
(b^{(k)})^{-1}(b^{(j)}_l)^{-1}b^{(k)}a^{\tilde{f}_{p-k+\mu+1}},a^{\tilde{f}_{p-k+\mu+2}},\ldots, a^{\tilde{f}_{p-1}}),
\]
where
\begin{align*}
 \tilde{f}_1= e^{(j)}_{l,k}-e^{(j)}_{l,k-\mu},\,&\ldots\,, \tilde{f}_{p-k}=e^{(j)}_{l,p-1}-e^{(j)}_{l,p-\mu-1},\\
 \tilde{f}_{p-k+1}&=-e^{(j)}_{l,p-\mu},\\
 \tilde{f}_{p-k+2}= e^{(j)}_{l,1}-e^{(j)}_{l,p-\mu+1},\, &\ldots\,,\,\tilde{f}_{p-k+\mu}=e^{(j)}_{l,\mu-1}-e^{(j)}_{l,p-1},\\
 \tilde{f}_{p-k+\mu+1}&=e^{(j)}_{l,\mu},\\
 \tilde{f}_{p-k+\mu+2}=e^{(j)}_{l,\mu+1}-e^{(j)}_{l,1},\, &\ldots\,,\, \tilde{f}_{p-1}=e^{(j)}_{l,k-1}-e^{(j)}_{l,k-\mu-1}.
\end{align*}
We proceed as in Case 1.

\bigskip

\underline{Case 3:} Suppose $\partial(x)=m>1$.

If $x\in \text{St}_G(1)$, then by~\cite[Lem.~3.10]{KT} there exists $j\in \{1,\ldots,p\}$ such that $\partial(\varphi_j(x))<m$, and we proceed by induction.

Now if $x\not\in \text{St}_G(1)$, then $x=ya^i$ for some $y\in \text{St}_G(1)$ and 
some $i\in \{1,\ldots, p-1\}$. 
We consider $\tilde{x}=[x,b^{(k)}_l]$, for some $k\in \{1,\ldots,p\}$ and $1\le l\le r_k$. The element $\tilde{x}$ has length at most $2m+2$. Then there is a $j\in \{1,\ldots,p\}$ such that $\partial(\varphi_j(\tilde{x}))<m$: indeed, this is clear by~\cite[Lem.~3.10]{KT}  apart from the case $p=3$ and $m=2$. So suppose $p=3$ and that $x$ is of length $2$. Then 
\[
x=a^{i_1}b^{(j_1)}a^{i_2}b^{(j_2)}a^{i_3}
\]
for $i_1,i_2,i_3\in \mathbb{F}_p$ with $i_1+i_2+i_3\ne 0$ and $b^{(j_k)}\in \langle \mathbf{b}^{(j_k)}\rangle$ for $j_k\in \{1,2,3\}$ with $r_{j_k}\ne 0$, where $k\in \{1,2\}$. Equivalently,
\[
x=a^{i_1+i_2+i_3} (b^{(j_1)})^{a^{i_2+i_3}}(b^{(j_2)})^{a^{i_3}}
\]
with $i_1+i_2+i_3\ne 0$. Now consider some $b^{(l)}\in \langle \mathbf{b}^{(l)}\rangle$ for $l\in \{1,2,3\}$ with $r_l\ne 0$. To simplify notation, we write $b=(b^{(j_1)})^{a^{i_2+i_3}}$, $c=(b^{(j_2)})^{a^{i_3}}$, $d=b^{(l)}$ and $i=i_1+i_2+i_3$. Then
\[
[x,d] = c^{-1} b^{-1} d^{- a^{i}} b c d.
\]
Under $\psi_1$, each of the directed automorphisms $b,c,d$ will project an element of length $1$ in exactly one coordinate. Since there are only three coordinates, and $[x,d]$ consists of the four directed automorphisms $b$, $c$, $d$ and $d^{a^i}$, it follows by the pigeonhole principle that there is a coordinate $j\in \{1,2,3\}$ such that $\partial(\varphi_j([x,d]))$ is at most $\lfloor 4/3\rfloor=1$.
\end{proof}


\subsection{Torsion groups}

It is well known that GGS-groups are torsion if and only if the components of the defining vector sum to zero modulo $p$. As the directed generators of a multi-GGS group commute, it is clear that a multi-GGS group is torsion if and only if for every defining vector, the components sum to zero modulo $p$. The same is true for the groups in~$\mathscr{C}$, however it was pointed out to us by G.\,A. Fern\'{a}ndez-Alcober that no explicit proof of this fact exists in the literature. We provide the necessary details here.

\begin{lemma}
Let $G=\langle a, \mathbf{b}^{(1)}, \ldots, \mathbf{b}^{(p)} \rangle
  \in \mathscr{C}$ be in standard form, associated to the numerical datum $\mathbf{E} = (\mathbf{E}^{(1)}, \ldots, \mathbf{E}^{(p)})$. Then $G$ is torsion if and only if 
  \[
  \sum_{\ell =1}^p  e^{(j)}_{i,\ell} \equiv_p 0
  \]
  for all $j\in \{1,\ldots,p\}$ and $1\le i\le r_j$.
\end{lemma}

\begin{proof}
The forward direction is obvious, based on the known result for GGS-groups. So we suppose that, for every defining vector, the components sum to zero modulo $p$. Let $g\in G$ be arbitrary. We will show that $g$ has finite order, by induction on the length $\partial(g)$. 

If $\partial(g)=0$, then $g=a^k$ for some $k\in\mathbb{F}_p$, and the result is clear. If $\partial(g)=1$, then $g\in G_j= \langle a,\mathbf{b}^{(j)} \rangle$ for some $j\in\{1,\ldots,p\}$, and the result is clear. Hence we suppose that $\partial(g)=n\ge 2$ and that the result holds for all elements $g'\in G$ with $\partial(g')<n$.

\underline{Case 1:} Suppose $g\in \text{St}_G(1)$. If, for all $j\in\{1,\ldots,p\}$ we have $\partial(\varphi_j(g))<\partial(g)$, then the result follows by induction. Therefore, we suppose otherwise, and hence, by~\cite[Lem.~3.10]{KT}, there is exactly one $j\in\{1,\ldots,p\}$ such that $\partial(\varphi_j(g))=\partial(g)$ and $\varphi_i(g)$ is a power of $a$ for every $i\ne j$. It follows that $\varphi_j(g)\in\text{St}_G(1)$. By~\cite[Lem.~3.10]{KT}, for all $k\in\{1,\ldots,p\}$ we have $\partial(\varphi_{jk}(g))<\partial(g)$, and we may proceed by induction.

\underline{Case 2:} Suppose $g\not\in \text{St}_G(1)$. Then we can write $g=a^kh$ for some $0\ne k\in\mathbb{F}_p$ and $h=\psi_1^{-1}((h_1,\ldots,h_p))\in \text{St}_G(1)$. Consider $g^p=\psi_1^{-1}((g_1,\ldots,g_p))$, where $g_l\equiv h_1\cdots h_p$ modulo $G'$ and $\partial(g_l)=\partial(g)$, for all $l\in\{1,\ldots,p\}$. Since 
\[
\text{St}_G(1)=\langle (b^{(j)}_i)^{a^k} \mid j\in\{1,\ldots,p\}, i\in\{1,\ldots,r_j\}, k\in\mathbb{F}_p \rangle
\]
the exponent sum of $a$ in the product $h_1\cdots h_p$ is zero modulo $p$, by our assumption. Hence $g_1,\ldots,g_p\in\text{St}_G(1)$, and we proceed as in Case 1.
\end{proof}


\section{Congruence subgroup property}

\subsection{Regular branch over $G'$}

We first prove one direction of Theorem~\ref{MainTheorem}(A.1):
\begin{lemma}\label{lem:no-CSP}
Let
  $G= \langle a,\mathbf{b}^{(1)}, \ldots, \mathbf{b}^{(p)} \rangle \in
  \mathscr{C}$
  be in standard form and such that $G$ is regular branch over $G'$. If  the defining vectors
  $\mathbf{E}^{(1)}, \ldots, \mathbf{E}^{(p)}$ are linearly dependent, then $G$ does not have the congruence subgroup property.
\end{lemma}
\begin{proof}
Let $J=\{j\in \{1,\ldots,p\}\mid r_j\ne 0\}$. We may without loss of generality replace directed generators by multiples of themselves, and by suitable products within their families.
Then, from the condition on the defining vectors, there is some $c\in \mathbf{b}^{(j)}$, for some $j\in J$, that can be expressed as 
\[
c\equiv b_{i_1}^{a^{i_1-j}}b_{i_2}^{a^{i_2-j}}\cdots b_{i_m}^{a^{i_m-j}} \quad \text{mod }\text{St}_G(2),
\]
for some $m\in \mathbb{N}$ and  $b_{i_1}\in \mathbf{b}^{(i_1)},\ldots ,b_{i_m}\in \mathbf{b}^{(i_m)}$ where $i_1,\ldots,i_m$ are pairwise distinct elements in $J\backslash\{j\}$. 

We  show that $G'$ is not a congruence subgroup, by recursively constructing elements
\[
t_n\in b_{i_1}b_{i_2}\cdots b_{i_m} G' \cap c\text{St}_G(n)
\]
for each $n\in \mathbb{N}$. The result then follows by the fact that $c\not \equiv b_{i_1}b_{i_2}\cdots b_{i_m}$ modulo $G'$ (cf. \cite[Prop.~3.9]{KT}).

To begin, suppose that $\psi_1(c)=(a^{e_j},\ldots, a^{e_{p-1}},c,a^{e_1},\ldots,a^{e_{j-1}})$. Then
\[
\psi_1(b_{i_1}^{a^{i_1-j}}b_{i_2}^{a^{i_2-j}}\cdots b_{i_m}^{a^{i_m-j}})=(a^{e_j},\ldots, a^{e_{p-1}},b_{i_1}b_{i_2}\cdots b_{i_m},a^{e_1},\ldots,a^{e_{j-1}}).
\]
We set 
\[
t_1=b_{i_1}b_{i_2}\cdots b_{i_m}
\]
and
\begin{align*}
t_2&=b_{i_1}^{a^{i_1-j}}b_{i_2}^{a^{i_2-j}}\cdots b_{i_m}^{a^{i_m-j}}\\
&=b_{i_1}b_{i_2}\cdots b_{i_m} [b_{i_1},a^{i_1-j}][b_{i_2},a^{i_2-j}]\cdots [b_{i_m},a^{i_m-j}]d
\end{align*}
where $d \in \gamma_3(G)$.
Now suppose $t_{n-1}\in b_{i_1}b_{i_2}\cdots b_{i_m}G'\cap c\,\text{St}_G(n-1)$. Set
\[
x_n:=\psi_1^{-1}((1,\ldots, 1,\underset{(p-j+1)\text{th coordinate}}{\underbrace{(b_{i_1}b_{i_2}\cdots b_{i_m})^{-1}t_{n-1}}},1,\ldots,1))\in G'.
\]
Then
\begin{align*}
t_n&=b_{i_1}^{a^{i_1-j}}b_{i_2}^{a^{i_2-j}}\cdots b_{i_m}^{a^{i_m-j}}x_n\\
&=\psi_1^{-1}((a^{e_j},\ldots, a^{e_{p-1}},t_{n-1},a^{e_1},\ldots,a^{e_{j-1}}))
\end{align*}
and thus
\[
\psi_1^{-1}(c^{-1}t_n)=(1,\overset{p-j}{\ldots},1,c^{-1}t_{n-1},1,\overset{j-1}{\ldots},1).
\]
Since $c^{-1}t_{n-1}\in \text{St}_G(n-1)$, we have  $c^{-1}t_n\in \text{St}_G(n)$, as required.
\end{proof}

We note that these groups also do not have the $p$-congruence subgroup property because the derived subgroup $G'$, which is of $p$-power index, does not contain any level stabiliser. Also, as mentioned in \cite{GUA2}, having the congruence subgroup property or the $p$-congruence subgroup property is independent of the branch action of the group.

We now prove the remaining direction of Theorem \ref{MainTheorem}(A.1) in two steps.

\begin{proposition}\label{pr:st-in-derived}
Let  $G= \langle a,\mathbf{b}^{(1)}, \ldots, \mathbf{b}^{(p)} \rangle \in
  \mathscr{C}$
  be in standard form and such that the defining vectors
  $\mathbf{E}^{(1)}, \ldots, \mathbf{E}^{(p)}$ are linearly independent. Write $r=r_1+\cdots +r_p$. Then $G'\ge \textup{St}_G(r+1)$.
\end{proposition}
\begin{proof}
This is in essence the same proof as for \cite[Prop.~6]{AlejJone}.
\end{proof}

\begin{proposition}
Let  $G= \langle a,\mathbf{b}^{(1)}, \ldots, \mathbf{b}^{(p)} \rangle \in
  \mathscr{C}$
  be as in Theorem \ref{MainTheorem}(A) and such that the defining vectors
  $\mathbf{E}^{(1)}, \ldots, \mathbf{E}^{(p)}$ are linearly independent. Then $G$ has the congruence subgroup property.
\end{proposition}

\begin{proof}
By~\cite[Prop.~2.4]{FAGUA}, it suffices to show that $G''$ contains a level stabiliser.

We have from Lemma \ref{second_derived} that
\[
\psi_1(G'')\ge \gamma_3(G)\times \overset{p}{\cdots}\times \gamma_3(G).
\]
From Proposition \ref{key}, we have 
\[
 \psi_1(\gamma_3(G))\ge G'\times \overset{p}{\cdots}\times G'.
\]
Hence 
\begin{align*}
G''&\ge \psi_2^{-1}(G'\times \overset{p^2}{\cdots}\times G')\\
&\ge \psi_2^{-1}(\text{St}_G(r+1)\times \overset{p^{2}}{\cdots}\times \text{St}_G(r+1))\\
&=\text{St}_G(r+3),
\end{align*}
making use of the previous proposition. 
\end{proof}

The above result gives many examples of non-torsion branch groups with the congruence subgroup property. Recall that a multi-EGS group $G$ is non-torsion if at least one of its defining vectors has exponents not summing to zero modulo $p$. Non-torsion branch groups with the congruence subgroup property furthermore are not locally extended residually finite (LERF), (cf. \cite[Ch.~3]{Francoeur}), where a group $G$ is said to be \emph{LERF} (or \emph{subgroup separable}) if every finitely generated subgroup of $G$ is closed in the profinite topology; equivalently, if every finitely generated subgroup is the intersection of finite index subgroups.

\begin{proof}[Proof of Theorem \ref{MainTheorem}(B.1)]
The method here is similar to the proof of Lemma~\ref{lem:no-CSP}. Let $G=\langle a,b^{(j)},  b^{(k)}\rangle$, for some distinct $j, k\in \{1,\ldots,p\}$, with  symmetric defining vectors $(e_1,\ldots,e_{p-1})$ and $(f_1,\ldots,f_{p-1})$ satisfying 
$e_i, f_i\in \{0,1\}$ with $e_i\ne f_i$ for all $1\le i\le p-1$.

We  show that $\gamma_3(G)$ is not a congruence subgroup, by recursively constructing elements
\[
t_n\in [b^{(j)},b^{(k)}] \gamma_3(G) \cap \text{St}_G(n)
\]
for each $n\in \mathbb{N}$.
The result then follows by Lemma~\ref{lem:exception}.

Let $t_1=t_2=[(b^{(j)})^{a^{j-k}},b^{(k)}]$ with
\[
\psi_1([(b^{(j)})^{a^{j-k}},b^{(k)}])=(1,\overset{p-k}\ldots,1,[b^{(j)},b^{(k)}],1,\ldots,1)\in \psi_1(\text{St}_G(2)).
\]

Now suppose $t_{n-1}\in [b^{(j)},b^{(k)}] \gamma_3(G) \cap \text{St}_G(n-1)$. Hence $t_{n-1}[b^{(k)},b^{(j)}]\in\gamma_3(G)$ and there exists $x_n\in \gamma_3(G)$ such that
\[
\psi_1(x_n)=(1,\overset{p-k}\ldots, 1,t_{n-1}[b^{(k)},b^{(j)}],1,\ldots,1).
\]
Then $t_n=x_n[(b^{(j)})^{a^{j-k}},b^{(k)}]$ satisfies
\[
\psi_1(t_n)=(1,\overset{p-k}\ldots, 1,t_{n-1},1,\ldots,1)\in \psi_1(\text{St}_G(n)),
\]
and  $t_n\equiv [b^{(j)},b^{(k)}]$ mod $\gamma_3(G)$, as required.
\end{proof}

Likewise the groups in Theorem~\ref{MainTheorem}(B) do not possess the $p$-congruence subgroup property.


\subsection{Regular branch over $\gamma_3(G)$}

\begin{proposition}
 Let $G$ be as in Theorem~\ref{MainTheorem}(C). Then $\textup{St}_G(5)\le \gamma_3(G)$.
\end{proposition}
\begin{proof}
We may assume that $G$ is not a GGS-group, as the result already holds by \cite{FAGUA}.
 Suppose $x\in \textup{St}_G(5)$ with $\psi_1(x)=(x_1,\ldots,x_p)$. Working modulo 
 $\gamma_3(G)$, we will show that  $x_s\equiv 1$ for all $1\le s\le p$. 
 
 Fix $s\in \{1,\ldots,p\}$ and write $h=x_s\in \textup{St}_G(4)$. Then $\psi_1(h)=(h_1,\ldots,h_p)$
with  $h_i\in \textup{St}_G(3)$ for all  $1\le i\le p$. 

 Fix $t\in \{1,\ldots,p\}$ and write $g=h_t\in \textup{St}_G(3)$. Then for 
 \[
 \psi_1(g)=(g_1,\ldots, g_p)
 \]
 we have $g_{i}\in \text{St}_G(2)$ for all 
 $1\le i\le p$. 
 
 For notational convenience, we will first prove the result for the EGS-group $G=\langle a,b,c\rangle$ with 
 symmetric defining vector
  $(e_1,\ldots,e_{p-1})$, where  $e_1=1$ (cf. \cite[Lem.~3.1]{KT}) with $b=b^{(1)}$ and $c=b^{(p)}$.
    The proof then generalises 
 easily to the wider class of multi-EGS groups  with a single symmetric defining vector.

 From the proof of \cite[Lem.~3.6]{Pervova} (which also applies in the symmetric defining vector case),
 it follows that 
 \[
  g_{i}\in (c^{-1}b)^{l_{i}}G',
 \]
 for $1\le i\le p$ with $l_i\in \mathbb{Z}$. Further, as $[b,c]\in \gamma_3(G)$, we have
 $1=c^{-p}b^p\equiv  (c^{-1}b)^p$ mod $\gamma_3(G)$. Hence we may assume 
 $l_i\in \mathbb{F}_p$ for $1\le i\le p$. 

 First 
 we note that $G'/\gamma_3(G)$ is normally generated by 
 $[a,b]$ and $[a,c]$. Therefore
 \[
  g_{i}\equiv (c^{-1}b)^{l_{i}}[a,b]^{k_{i}}[a,c]^{m_{i}}
  \quad\text{mod }\gamma_3(G)\\
   \]
 where $m_{i},k_{i}\in \mathbb{F}_p$ with $1\le i\le p$.

  Next observe that 
 \begin{align*}
 \psi_1([b,c])&=([a,c],1,\ldots,1,[b,a])\\
   \psi_1([b,b^{a}])&=([a,b],1,\ldots, 1,[b,a])\\
    \psi_1([b,b^{a}]^a)&=([b,a],[a,b],1,\ldots, 1)\\
  &\,\,\,\vdots \\
    \psi_1([b,b^{a}]^{a^{p-1}})&=(1,\ldots, 1,[b,a], [a,b]),  
 \end{align*}
 and similarly for $\psi_1([c,c^a]^{a^j})$ for $0\le j\le p-1$. Indeed we also have 
 \begin{equation}\label{eq:final}
  \psi_1^{-1}((1,\ldots, 1, [b,a][a,c]))\in \gamma_3(G).
 \end{equation}

 Hence, 
 noting that $G$ is regular branch over $\gamma_3(G)$, we have
\begin{equation}\label{g}
\psi_1(g)\equiv \big((c^{-1}b)^{l_{1}},\, \ldots\,, \, (c^{-1}b)^{l_{p-1}} \, ,\,(c^{-1}b)^{l_{p}}[a,b]^{k} \big)
  \quad\text{mod }\psi_1(\gamma_3(G))
\end{equation}
for $k\in\mathbb{F}_p$. We claim that $k=0$. Indeed,  if $g\in G$ is such that (\ref{g}) holds, then there exists $g'\in G$ such that $\psi_1(g')= \big((c^{-1}b)^{l_{1}},\, \ldots\,, \, (c^{-1}b)^{l_{p-1}} \, ,\,(c^{-1}b)^{l_{p}}[a,b]^{k} \big)$. Then, multiplying $g'$ with 
\[
(c^{-1}b^a)^{l_{1}}(c^{-a}b^{a^2})^{l_{2}}\cdots (c^{-a^{p-1}}b)^{l_{p}} 
\]
gives the element $\psi_1^{-1}((1, \ldots, 1, [a,b]^{k}))\in G$,
which would imply that $G$ is regular branch over $G'$ if $k\ne 0$, a contradiction.
Thus
\begin{align*}
 \psi_1(g)&\equiv \big( (c^{-1}b)^{l_1} ,\ldots,  (c^{-1}b)^{l_p}\big) \quad \text{mod }\psi_1(\gamma_3(G)).
\end{align*}

Then the fact 
\[
 \psi_1^{-1}((c^{-1}b,1,\ldots,1,b^{-1}c))= c^{-1}b^a b^{-1}c^{a^{-1}}\equiv [b,a][a,c]\quad\text{mod }\gamma_3(G)
\]
gives that
\begin{align*}
 g&\equiv \psi_1^{-1}(( 1 ,\ldots,  1, (c^{-1}b)^{l}))[b,a]^{l-l_p}[a,c]^{l-l_p} \quad \text{mod }\gamma_3(G)\\
 &\equiv (c^{-1}b^a)^{la^{-1}} [b,a]^{l-l_p}[a,c]^{l-l_p} \quad \text{mod }\gamma_3(G),
\end{align*}
where $l=l_1+\cdots+l_p$.

Recall that $g=h_t$, for $t\in \{1,\ldots,p\}$. Thus, by (\ref{eq:final}), for some $\alpha_1,\ldots, \alpha_p\in \mathbb{F}_p$,
\begin{align*}
 \psi_1(h)&\equiv \big( (c^{-1}b^a)^{\alpha_1} ,\ldots,(c^{-1}b^a)^{\alpha_p } \big)\quad \text{mod }\psi_1(\gamma_3(G))\\
 &\equiv \big( (c^{-1}b[b,a])^{\alpha_1} ,\ldots,(c^{-1}b[b,a])^{\alpha_p } \big)\quad \text{mod }\psi_1(\gamma_3(G))\\
 &\equiv \big( (c^{-1}b)^{\alpha_1} ,\ldots,(c^{-1}b)^{\alpha_{p-1} },(c^{-1}b)^{\alpha_p } [b,a]^{\alpha_1+\cdots+\alpha_p}\big)\quad \text{mod }\psi_1(\gamma_3(G)).
\end{align*}
In order to avoid getting the same sort of contradiction, we must have $\alpha_1+\cdots+\alpha_p= 0$. Hence
\begin{align*}
 \psi_1(h) &\equiv \big( (c^{-1}b)^{\alpha_1} ,\ldots,(c^{-1}b)^{\alpha_{p-1} },(c^{-1}b)^{\alpha_p } \big)\quad \text{mod }\psi_1(\gamma_3(G))\\
 &\equiv \big( 1,\ldots,1,(c^{-1}b)^{\alpha_1+\cdots+\alpha_p }\big)\psi_1(([b,a][a,c])^{\beta})\quad \text{mod }\psi_1(\gamma_3(G))\\
 &\equiv \psi_1(([b,a][a,c])^{\beta})\quad \text{mod }\psi_1(\gamma_3(G)),
\end{align*}
for some $\beta\in \mathbb{F}_p$.

Recall that $h=x_s$, for $x\in \{1,\ldots,p\}$. Thus,  for some $\beta_1,\ldots, \beta_p\in \mathbb{F}_p$,
\begin{align*}
 \psi_1(x)&\equiv \big( ([b,a][a,c])^{\beta_1} ,\ldots,([b,a][a,c])^{\beta_p} \big)\quad \text{mod }\psi_1(\gamma_3(G))\\
 &\equiv 1\quad \text{mod }\psi_1(\gamma_3(G)).
\end{align*}
Thus we have that $x\in \gamma_3(G)$.

For the general case, when $G$ has more than 2 directed generators, for $g_i\in \text{St}_G(2)$, we have 
\[
g_i\equiv (b^{(j_1)})^{\beta_1}(b^{(j_2)})^{\beta_2}\cdots (b^{(j_r)})^{\beta_r} \quad \text{mod }G',
\]
for some $r\in \mathbb{N}_{\ge 3}$, $j_1,\ldots,j_r\in \{1,\ldots, p\}$ with $j_k\ne j_{k+1}$ for $1\le k\le r-1$, and $\beta_1,\ldots, \beta_r\in \mathbb{F}_p\backslash\{0\}$ such that $\beta_1+\cdots +\beta_r = 0$.

Hence, using Proposition~\ref{key},  for each $1\le i\le p$, we have
\[
 g_{i} \equiv \prod_{k\ne l,\, r_k,r_l\ne 0} g_{i}^{(kl)} \quad \text{mod }\gamma_3(G),
\]
where $g_{i}^{(kl)}\in \text{St}_{G_{kl}}(2)$ for $G_{kl}=\langle a, b^{(k)},b^{(l)}\rangle$. One then applies the above argument for each $g_{i}^{(kl)}$.
\end{proof}

\begin{proof}
[Proof of Theorem \ref{MainTheorem}(C)]
Let $G$ be as specified. 
We have $\gamma_3(G)\ge \text{St}_G(1)'$ by Proposition \ref{key}, and therefore 
$\gamma_4(G)\ge \gamma_3(\text{St}_G(1))=\psi_1^{-1}(\gamma_3(G)\times \overset{p}\cdots \times \gamma_3(G))$. By the previous 
proposition, we see that $\psi_1(\gamma_4(G))\ge \text{St}_G(5)\times \overset{p}\cdots \times \text{St}_G(5)
$.
The result now follows
from the fact that $\psi_1(\gamma_3(G)')\ge \gamma_4(G)\times \overset{p}\cdots \times \gamma_4(G)$ (compare Lemmata~\ref{notderivedsubgroup} and \ref{lem:subdirect}(ii)), making use of~\cite[Prop.~2.4]{FAGUA}.
\end{proof}

\section{The profinite completion}

Now for the above groups $G$ which have the congruence subgroup property, we have that the profinite completion of $G$ coincides 
with the closure of $G$ in $\text{Aut}(T)$. This section describes the profinite completion of the regular branch multi-EGS 
groups without the congruence 
subgroup property; compare Theorem \ref{MainTheorem}(A)  and (B).

Let $n\in \mathbb{N}$. Denote by $K_n$ the subgroup of $G$ satisfying 
$\psi_n(K_n)=G' \times \overset{\,\,\,p^n}{\cdots} \times G'$, and by $M_n$ the subgroup of $G$ satisfying 
$\psi_n(M_n)=\gamma_3(G) \times \overset{\,\,\,p^n}{\cdots} \times \gamma_3(G)$.
Parts (A.2) and (B.2) of Theorem \ref{MainTheorem}  are deduced  from the following:

\begin{theorem}\label{Theorem 3.2}
Let $G$ be a multi-EGS group that is regular branch over $G'$. Then for every non-trivial normal subgroup $N$ in $G$, there is an $n\in\mathbb{N}$ such that $N$ contains $M_n$. Furthermore, if $G\not \in \mathscr{E}$, then $N$ contains $K_n$.
\end{theorem}
\begin{proof} It suffices to prove the theorem for normal closures of non-trivial elements $x\in G$.
Hence let $N$ denote the normal closure in $G$ of a non-trivial element $x\in G$.
By Proposition \ref{Proposition 2.1}, there is a vertex $v$ such that $\text{St}_N(v)_v=G$. 
Observe that 
\[
[\text{St}_N(v),\text{Rist}_{G}(v)]\le \text{Rist}_N(v).
\]
Since $G$ is regular branch over $G'$, we have $\text{Rist}_{G}(v)_v\ge G'$.
Hence
\[
\text{Rist}_N(v)_v \ge \varphi_v([\text{St}_N(v),\text{Rist}_{G}(v)])\ge \gamma_3(G),
\]
and $N\ge M_{|v|}$. By conjugation, the first statement follows.

Suppose now that $G\not \in \mathscr{E}$. By Proposition \ref{key}, we have $\gamma_3(G)\ge \text{St}_G(1)'=K_1$. 
Therefore conjugation gives $N\ge K_{|v|+1}$,
as required.
\end{proof}

\begin{proof}[Proof of Corollary~\ref{justinfinite}]
For the groups with the congruence subgroup property, the result follows from \cite[Prop.~2.4]{FAGUA}. For the remaining groups, the result follows from Theorem \ref{Theorem 3.2} and from the fact that $\gamma_3(G)\times \overset{p^n}{\cdots} \times \gamma_3(G)$ has
finite index in $G$ for every $n\in\mathbb{N}$.
\end{proof}

We observe below that all branch multi-EGS groups have the weak congruence subgroup property, that is, every finite-index subgroup contains $\text{St}_G(n)'$, for some $n\in \mathbb{N}$; compare \cite{Segal}.

\begin{corollary}
Let $G\in \mathscr{C}\backslash \mathscr{G}$ be a multi-EGS group. Then $G$ has the weak congruence subgroup property.
\end{corollary}

\begin{proof}
 Indeed, if $G$ has the congruence subgroup property then the result is clear. So we first suppose that $G\notin\mathscr{E}$ is regular branch over $G'$ and suppose $N\trianglelefteq G$ is a normal subgroup of finite index. Then, by the previous result, there exists an $n\in \mathbb{N}$ such that $N\ge K_n$. 
 
 We show by induction that $K_n\ge \text{St}_G(n)'$ for all $n\in \mathbb{N}$. For $n=1$, we in fact have $K_1=\text{St}_G(1)'$. Next, suppose the statement is true for some $n\in \mathbb{N}$. Then 
 \[
     \psi_1(K_{n+1})=K_n\times \overset{p}\cdots\times K_n 
     \ge \text{St}_G(n)'\times \overset{p}\cdots\times\text{St}_G(n)'
     \ge \psi_1(\text{St}_G(n+1)').
 \]
 Hence the result for this case.
 
 Lastly, we suppose that $G\in\mathscr{E}$ and let $N\trianglelefteq G$ is a normal subgroup of finite index. By Theorem~\ref{Theorem 3.2}, there exists an $n\in \mathbb{N}$ such that $N\ge M_n$. Since $\gamma_3(G)\ge G''$, 
 and $G''\ge \text{St}_G(3)'$ by Proposition~\ref{pr:st-in-derived}, we obtain
 \[
 \psi_n(M_n)\ge G''\times \overset{p^n}\cdots \times G'' \ge \text{St}_G(3)' \times \overset{p^n}\cdots\times \text{St}_G(3)'\ge \psi_n(\text{St}_G(n+3)').
 \]
 Hence the result.
\end{proof}


\section{Multi-EGS groups with constant defining vector}

Here we prove that a multi-EGS group $G\in \mathscr{G}$ with constant defining vector is weakly regular branch  
but not branch. The corresponding result for a GGS-group with constant defining vector was proved in \cite{FAGUA}, hence we
 assume in the sequel that $G$ is not 
 of the form $\langle a,b^{(j)}\rangle$, for some $j\in\{1,\ldots,p\}$.

Let $G=\langle a, b^{(1)},\ldots, b^{(p)} \rangle \in \mathscr{G}$ be in standard form. Let $J$ be the set of indices 
$j\in \{1,\ldots,p\}$ such that $r_j\ne 0$. Now we set $K=\langle b^{(j)}a^{-1} \mid j\in J\rangle ^G$.
Note that upon replacing $b^{(j)}$ with a suitable power, we may assume that 
$\psi_1((b^{(j)})^{a^j})=(b^{(j)},a,\ldots,a)$ for all $j\in J$. We write $y^{(j)}_i=(b^{(j)}a^{-1})^{a^i}$ for $j\in J$ and 
$i\in \{0,1,\ldots,p-1\}$.

\begin{lemma}\label{structure}
 Let $G\in \mathscr{G}$ with $K$ as above. Then
 
 \begin{enumerate}
 \item[(i)] $K=\langle y^{(j)}_i \mid 0\le i\le p-1,\, j\in J\rangle$,
 
 \item[(ii)] $G'\le K$ and $|G:K|=p$,
 
 \item[(iii)] $K'\times \overset{p}\cdots \times K' \le \psi_1(K')\le \psi_1(G')\le K\times \overset{p}\cdots \times K$,
 
 \item[(iv)] $K'$ is generated, modulo $\psi_1^{-1}(K'\times \overset{p}\cdots \times K')$, by $[y^{(j)}_i,y^{(k)}_l]^c$ where $0\le i,l\le p-1$, and $j,k\in J$ with $c\in 
 \langle b^{(j)} \mid j\in J\rangle $.
 \end{enumerate}
\end{lemma}

\begin{proof}
 (i) It suffices to check that $\langle y^{(j)}_i \mid 0\le i\le p-1, \, j\in J\rangle$ is normal in $G$.
 Now this follows from the fact  $
 (y^{(j)}_i)^{b^{(k)}}= (y^{(j)}_{i+1})^{y^{(k)}_1}$ for $0\le i\le p-1$  and $j,k\in J$.

 (ii) First, to show that $G'\le K$, it suffices to observe that $K$ is normal in $G$, and that for $j,k\in J$,
 \[[a,b^{(j)}a^{-1}]=[a,b^{(j)}]^{a^{-1}}\in K\quad \text{and}\quad 
 [a^{-1},b^{(k)}][b^{(k)},b^{(j)}a^{-1}]=[b^{(k)},b^{(j)}]^{a^{-1}}\in K.
 \]
 Hence $|G:K|=|G/G' : K/G'|$ and as $a\equiv b^{(j)}$ modulo $K$ for all $j\in J$, it follows that $|G:K|=p$, as required.

(iii) We observe that for distinct $j, k\in J$,
 \[
  \psi_1([a,b^{(j)}])=(1,\ldots, 1,a^{-1}b^{(j)}, (b^{(j)})^{-1}a,1,\ldots,1)\in K\times \cdots \times K
 \]
 and
 \begin{align*}
  \psi_1(&[b^{(j)},b^{(k)}])\\
  &=(1,\ldots, 1, [b^{(j)},a],1,\ldots,1,[a,b^{(k)}],1,\ldots,1)\in G'\times \cdots \times G'\\
  &\qquad \qquad \qquad \qquad \qquad \qquad \qquad \qquad \qquad \,\,\,\quad\le K\times \cdots \times K.
 \end{align*}
 Since $G'=\langle [a,b^{(j)}],[b^{(j)},b^{(k)}] \mid j ,k\in J,\, j\ne k\rangle^G$ and $K$ is normal in $G$, it then follows that 
 $\psi_1(G')\le K\times \cdots \times K$.
 
 For the remaining inclusion $\psi_1(K')\ge K'\times \cdots \times K'$, it follows from \cite{FZ} that 
 $([y^{(j)}_i,y^{(j)}_l],1,\ldots,1)\in \psi_1(K')$ for $j\in J$ and $0\le i,l\le p-1$. Hence it suffices to show  that
 $([y^{(j)}_i,y^{(k)}_l],1,\ldots,1)\in \psi_1(K')$ for distinct $j, k\in J$ and $0\le i,l\le p-1$. 
 
 Let $j, k\in J$ be distinct elements. Without loss of generality we suppose that $j>k$. We observe that 
 $y^{(j)}_0(y^{(k)}_0)^{-1}=b^{(j)}(b^{(k)})^{-1}\in K$ and 
 \[
 \psi_1(y^{(j)}_0(y^{(k)}_0)^{-1})=(1,\ldots,1,y^{(j)}_0,1,\ldots, 1,
 (y^{(k)}_0)^{-1},1,\ldots,1).
 \]
 For $0\le i\le p-1$, let $g_i$ be a conjugate of $y^{(j)}_0(y^{(k)}_0)^{-1}$ 
 by an appropriate $(b^{(j)})^{ia^*}$ such that 
 $\psi_1(g_i)=(1,\ldots,1,y^{(j)}_i,1,\ldots, 1,
 (y^{(k)}_i)^{-1},1,\ldots,1)$. Then, using   $h:=(g_l^{-1})^{a^{k-j}}$, we obtain
 \[
  \psi_1([g_i,h])=(1,\ldots, 1, [y^{(j)}_i, y^{(k)}_l],1,\ldots,1),
 \]
and the result follows.

(iv) We consider $[y^{(j)}_i, y^{(k)}_l]^g$ for $g\in G$ and write $g=ha^{\alpha}c$ with $h\in G'$ and 
$c\in  \langle b^{(j)} \mid j\in J\rangle $. Thus
\begin{align*}
 [y^{(j)}_i, y^{(k)}_l]^g&= ([y^{(j)}_i, y^{(k)}_l][y^{(j)}_i, y^{(k)}_l,h])^{a^{\alpha}c}\\
 &\equiv [y^{(j)}_i, y^{(k)}_l]^{a^{\alpha}c}\quad \text{mod }\psi_1^{-1}(K'\times \cdots \times K')\\
 &= [y^{(j)}_{i+\alpha}, y^{(k)}_{l+\alpha}]^{c},
\end{align*}
where the equivalence come from the fact $\psi_1(G'')\le K'\times \cdots \times K'$.
\end{proof}

In particular, Lemma \ref{structure}(iii) shows that $G$ is weakly regular branch over $K'$.

\smallskip

Let $j\in J$. For $G_j=\langle a,b^{(j)}\rangle$ we set $K_j=\langle b^{(j)}a^{-1}\rangle ^{G_j}$.

\begin{lemma}\label{Kcommutator}
Let $j\in J$. Then $K'\cap K_j=K_j'$.
\end{lemma}

\begin{proof}
As $x$ is in $K'$, this implies that we can write $x$ as a product of conjugates of the commutators $[y^{(m)}_i,y^{(k)}_l]$, 
where $0\le i,l\le p-1$ and $m,k\in J$. If this product involves  elements $y^{(m)}_i$ for $m\ne j$ (in other words, 
if $x\not \in K_j'$), then by~\cite[Lem.~3.10]{KT} there exists a vertex $u$ such that the section of $x$ at $u$ is $b^{(m)}$. However, this element cannot then be in $K_j$. Hence the result.
\end{proof}

\begin{lemma}\label{order}
The elements $y^{(j)}_i$, for $j\in J$ and $0\le i\le p-1$, are of infinite order in $G$ and also in $K/K'$.
\end{lemma}

\begin{proof}
 The first part follows as in \cite[Lem.~3.3]{FAGUA}. For the second, note that if $(y^{(j)}_i)^n\in K'$ for some 
 $n\in\mathbb{N}$ then $(y^{(j)}_i)^n\in K'\cap K_j=K_j'$ by Lemma~\ref{Kcommutator}. The result now follows from the fact that $K_j/K_j'$ is 
 torsion-free (cf. \cite[Prop.~3.4]{FAGUA}).
\end{proof}

\begin{lemma}\label{star}
 Let $j\in J$. For every $g\in K_j$ we have $gg^a\cdots g^{a^{p-1}}\in K_j'$. Furthermore, if $h\in K_j'$ with 
 $\psi_1(h)=(h_1,\ldots,h_p)$, then $h_p\cdots h_1\in K_j'$.
\end{lemma}

\begin{proof}
 This is just \cite[Lem.~4.3 and 4.4]{FZ}.
\end{proof}

\begin{proof}[Proof of Theorem \ref{constant}]
 We proceed as in~\cite[Thm.~3.7]{FAGUA}. Let 
 \[
 L=\psi_1^{-1}(K'\times \overset{p}\cdots \times K').
 \]
 By Lemma \ref{structure}(iii), we have $L\subseteq \text{Rist}_{G'}(1)$. 
 We prove that equality holds by considering $g\in \text{Rist}_{G'}(x)$ for $x\in X$ and showing that $g\in L$. 
 From the definition of the rigid stabiliser of a vertex, all components of $\psi_1(g)$ are trivial, except possibly the 
 component corresponding to position $x$, say $h$. We observe that $h\in K$, since $\psi_1(G')\subseteq K\times \overset{p}\cdots 
 \times K$ by Lemma \ref{structure}(iii). As $g\in K$, we may write $g$ as $(\prod_{j\in J}g^{(j)})\cdot k$ with 
 $g^{(j)}\in K_j$ and $k\in K'$. Likewise, writing $\tilde{g}=gg^a\cdots g^{a^{p-1}}$, we obtain
 \[
  \tilde{g}=\big(\prod_{j\in J} \tilde{g}^{(j)} \big)\cdot \tilde{k}
 \]
for $\tilde{k}\in K'$ and $\tilde{g}^{(j)}=g^{(j)}(g^{(j)})^a\cdots (g^{(j)})^{a^{p-1}}$. From Lemma \ref{star}, we have 
$\tilde{g}^{(j)}\in K_j'\le K'$ for all $j\in J$, hence $\tilde{g}\in K'$. Now $\psi_1(\tilde{g})=(h,\ldots,h)$ and we may write
$h=(\prod_{j\in J}h^{(j)})\cdot l$ for $l\in K'$ and $h^{(j)}\in K_j$. Thus
\[
\psi_1(\tilde{g})=(h,\ldots,h)\equiv \prod_{j\in J}(h^{(j)},\ldots,h^{(j)})\quad \text{mod }K'\times \overset{p}\cdots \times K'.
\]
For each $j\in J$, write $k^{(j)}$ for $\psi_1^{-1}(h^{(j)},\ldots,h^{(j)})$. Clearly $k^{(j)}\in K'\cap K_j=K_j'$. From \cite[Prop.~3.4]{FAGUA}, we have that $K_j/K_j'$ is torsion-free and hence, taking the second part of Lemma~\ref{star} into account, it follows that $h^{(j)}\in K_j'\le K'$ 
for all $j\in J$. Thus $\psi_1(g)\in K'\times \overset{p}\cdots \times K'$ and $g\in L$, as required.

Now suppose on the contrary that $G$ is a branch group. Then $|G:\text{Rist}_G(1)|$ is finite and by 
\cite[Lem.~3.6]{FAGUA}, we have $|G':\text{Rist}_{G'}(1)|=|G':L|$ is also finite. As $L\le K'$, it follows that $G/K'$ 
is finite. However Lemma \ref{order} implies that $K/K'$ is infinite, a contradiction.
\end{proof}


\end{document}